\documentclass[11pt]{amsart}
\input epsf
\usepackage{latexsym}
\usepackage{amssymb,amsmath,amsthm,amscd, amssymb,
graphicx, enumerate, verbatim}
\usepackage[all]{xy}
\numberwithin{equation}{section}
\newtheorem{cor}[equation]{Corollary}
\newtheorem{lem}[equation]{Lemma}
\newtheorem{prop}[equation]{Proposition}
\newtheorem{thm}[equation]{Theorem}

\newtheorem{quest}[equation]{Question}

\newtheorem{defn}[equation]{Definition}

\newtheorem{Example}[equation]{Example}

\newtheorem{remark}[equation]{Remark}
\newenvironment{rmk}{\begin{remark}\rm}{\end{remark}}
\def\co{\colon\thinspace}

\newcommand{\vol}{\mbox{vol}}

\newcommand{\e}{\epsilon}

\def\G{\Gamma}
\def\g{\gamma}
\def\L{\Lambda}

\def\d{\partial}

\def\r{\rho}

\def\l{\lambda}

\def\Z{\mathbb{Z}}
\def\S1{\bf S^1}

\newcommand{\ddr}{{\d_r}}
\newcommand{\R}{{\mathbb R}}
\newcommand{\C}{{\mathbb C}}
\newcommand{\rhn}{{\mathbf{H}^n}}
\newcommand{\chn}{{\mathbf{CH}^n}}

\newcommand{\chtwo}{{\mathbf{CH}^2}}
\newcommand{\chm}{{\mathbf{CH}^{n-1}}}
\newcommand{\rhm}{{\mathbf{H}^{n-2}}}

\newcommand{\rhtwo}{{\mathbf{H}^{2}}}
\newcommand{\cV}{{\mathcal V}}
\newcommand{\cH}{{\mathcal H}}
\begin{document}

\abovedisplayskip=6pt plus3pt minus3pt
\belowdisplayskip=6pt plus3pt minus3pt

\title[Real hyperbolic hyperplane complements]
{\bf Rigidity and relative hyperbolicity of real hyperbolic
hyperplane complements}

\rm

\author{Igor Belegradek}
\address{Igor Belegradek\\School of Mathematics\\ Georgia Institute of
Technology\\ Atlanta, GA 30332-0160}\email{ib@math.gatech.edu}
\thanks{This work was partially supported by the NSF grant
\# DMS-0503864.}
\date{}
\begin{abstract} 
For $n>3$ we study spaces obtained from finite 
volume complete real hyperbolic $n$-manifolds by removing a compact 
totally geodesic submanifold of codimension two.
We prove that their fundamental groups are relative hyperbolic,
co-Hopf, biautomatic, residually hyperbolic, not K\"ahler,
not isomorphic to lattices in virtually connected real Lie groups,
have no nontrivial subgroups with property (T), 
have finite outer automorphism groups,
satisfy Mostow-type Rigidity,
have finite asymptotic dimension and rapid decay property,
and satisfy Baum-Connes conjecture. 
We also characterize those lattices in real Lie groups 
that are isomorphic to relatively hyperbolic groups.  
\end{abstract}\maketitle
\begin{center}\it
Dedicated to Thomas Farrell and Lowell Jones\rm
\end{center}

\section{Introduction}

Let $M$ be a (connected) complete finite volume
locally symmetric Riemannian manifold of negative sectional curvature, 
and let $S$ be a (possibly disconnected) 
compact totally geodesic submanifold of $M$ of codimension two.
As noted in Lemma~\ref{lem: morse}, 
the manifold $M\setminus S$ is aspherical, and 
Theorems~\ref{thm: intro-gr-theoretic-cor}--\ref{thm: intro-mostow}
indicate that the group $\pi_1(M\setminus S)$ shares
various rigidity properties with lattices in rank one semisimple Lie groups.
There are good reasons to focus on codimension two:
if $S$ had codimension one, then $M\setminus S$ would 
be a familiar object, namely,  
the interior of a compact real hyperbolic manifold with
totally geodesic boundary obtained from $M$ by cutting open along $S$, 
while if $S$ had codimension $>2$, then $M\setminus S$ 
would no longer be aspherical, and the groups
$\pi_1(M\setminus S)$ and $\pi_1(M)$ would be isomorphic. 

It is known that the pair $(M,S)$ is modelled on 
$(\chtwo, \rhtwo)$, $(\rhn, \rhm )$, or $(\chn, \chm)$, 
where $\chn$, $\rhn$ denotes
the hyperbolic $n$-space over $\C$, $\R$, respectively.
The purpose of this paper is to study various topological 
and group-theoretic properties of $M\setminus S$ 
in the simplest case
when $(M,S)$ is modelled on $(\rhn, \rhm )$.
Most results proved here extend to the case $(\chn, \chm)$, 
yet some of the proofs become much harder, and this case 
shall be treated in~\cite{Bel-ch-warp}. 
While I believe that the arguments in the case $(\chn, \chm)$ can be modified
to hold in the ``exceptional'' case $(\chtwo, \rhtwo)$, 
this remains to be checked.

Clearly $M\setminus S$ can be identified 
with the interior of a compact smooth manifold $N$ that
is obtained from $M$ by removing a tubular neighborhood of $S$
and chopping off all cusps (in case $M$ is noncompact). 
There are two kinds of components of $\d N$:
compact flat manifolds appearing as cusp cross-sections of $M$
(appearing if $M$ is noncompact),
and virtually trivial circle bundles over components of $S$, 
which are locally isometric to the Riemannian product of $S$ and 
a round circle. 
It follows from warped product computations of Heintze and
Schroeder~\cite{Sch} that $N$ admits a Riemannian metric of $\sec\le 0$
such that $\sec< 0$ on the interior of $N$, and the boundary
$\d N$ is totally geodesic (see Remark~\ref{rmk: HSF metrics}). 
In particular, the group $\pi_1(N)$ is $CAT(0)$. 
Furthermore, by compactness of $N$
the Heintze-Schroeder metric is $A$-regular, so deep work
of Farrell-Jones~\cite{FJ-A-reg} implies that
that $\pi_1(N)$ satisfies Borel's Conjecture.

Results of this paper are summarized below.

\begin{thm}\label{thm: intro-gr-theoretic-cor}
Let $M$ be a complete finite volume real hyperbolic $n$-manifold, 
and $S$ is a compact totally geodesic submanifold of codimension two.
Then \newline
$\mathrm{(1)}$ 
$\pi_1(N)$ is non-elementary (strongly) relatively hyperbolic, 
where the
peripheral subgroups are fundamental groups of the components
of $\d N$.
\newline
$\mathrm{(2)}$ the relatively hyperbolic boundary of $\pi_1(N)$ 
is the $(n-1)$-sphere.\newline
$\mathrm{(3)}$ if $n>2$, then 
$\pi_1(N)$ does not split as an amalgamated product
or an HNN-extension over subgroups of  
peripheral subgroups of $\pi_1(N)$, or over $\Z$.\newline
$\mathrm{(4)}$  if $n>2$, then $\pi_1(N)$ is co-Hopf.\newline
$\mathrm{(5)}$ for any finite subset $S\subset\pi_1(N)$  there is a homomorphism
of $\pi_1(N)$ onto a non-elementary hyperbolic group  that is injective on $S$.\newline
$\mathrm{(6)}$ $\pi_1(N)$ satisfies Strong Tits Alternative.\newline
$\mathrm{(7)}$ $\pi_1(N)$ is biautomatic.\newline
$\mathrm{(8)}$  No nontrivial subgroup of $\pi_1(N)$ has Kazhdan property 
\textup{(T)}. \newline
$\mathrm{(9)}$ $\pi_1(N)$ is not a K\"ahler group.
\newline
$\mathrm{(10)}$ $\pi_1(N)$ has finite asymptotic dimension.\newline
$\mathrm{(11)}$ $\pi_1(N)$ has rapid decay property.\newline
$\mathrm{(12)}$ $\pi_1(N)$ satisfies Baum-Connes conjecture.\newline
$\mathrm{(13)}$ if $n>3$, then
$\pi_1(N)$ is not isomorphic to the fundamental group of a 
complete negatively pinched Riemannian manifold.\newline
$\mathrm{(14)}$ 
if $\pi_1(N)$ is isomorphic to a lattice $\Lambda$
in a real Lie group $G$, then the identity component $G_0$ 
of $G$ is compact, $\Lambda\cap G_0$ is trivial, and $\Lambda$ 
projects isomorphically onto a finite index subgroup of $G/G_0$.
\end{thm}

The following is a version of Mostow rigidity.

\begin{thm}\label{thm: intro-mostow} 
For $n>2$ and $i=1,2$, suppose that $M_i$ is a complete finite volume
real hyperbolic $n$-manifold, and $S_i$ is a compact
totally geodesic submanifold of codimension two. Then
any homotopy equivalence $f\co M_1\setminus S_1\to M_2\setminus S_2$
induces an isometry $\iota_f\co M_1\to M_2$
taking $S_1$ to $S_2$ such that the restriction
$\iota_f\co M_1\setminus S_1\to M_2\setminus S_2$
is homotopic to $f$. Moreover, $\iota_f$ is uniquely
determined by the homotopy class of $f$.
\end{thm}

\begin{cor}\label{cor: intro-out}
For $n>2$, if $M$ is a complete finite volume
real hyperbolic manifold, and $S$ is a compact
totally geodesic submanifold of codimension two, then
the correspondence $f\to \iota_f$ induces an isomorphism
of the outer automorphism group of $\pi_1(M\setminus S)$
onto the group of isometries of $M$ that map $S$ to itself. 
In particular, the outer automorphism group of $\pi_1(M\setminus S)$
is finite.
\end{cor}

A key technical result of this paper is part (1) of
Theorem~\ref{thm: intro-gr-theoretic-cor}.
Similarly to Heintze-Schroeder computation, we show that 
$M\setminus S$ admits a finite volume complete Riemannian metric
whose sectional curvature is bounded above by a negative constant.
Then (1) and (2) of Theorem~\ref{thm: intro-gr-theoretic-cor}
follow from the closer look at properties of the metric;
in fact it satisfies Gromov's definition of relative hyperbolicity,
as elaborated in~\cite{Bow-rel}.
It is worth mentioning that previously
Fujiwara~\cite{Fuj} constructed
a finite volume complete Riemannian metric
on $M\setminus S$
with sectional curvatures within $[-1,0)$,
but it is unclear to me how to use Fujiwara's metric
to prove relative hyperbolicity.

Alternatively, part (1) of Theorem~\ref{thm: intro-gr-theoretic-cor}
can be deduced by combining results of Heintze, Schroeder~\cite{Sch},
Kapovich-Leeb~\cite{KapLee}, and Dru{\c{t}}u-Osin-Sapir~\cite{DOS}.
Details are given in Appendix~\ref{sec: another proof of rh},
but briefly the universal cover $\tilde N$ of $N$ with
Heintze-Schroeder's metric is a $CAT(0)$ space 
that has negative curvature away from $\d \tilde N$,
and path components of $\d \tilde N$ are isometric either 
to $\mathbb R^{n-1}$ or $\mathbb R\times\rhm$.
Then an argument of Kapovich-Leeb~\cite{KapLee} shows that
$\tilde N$ is asymptotically tree-graded with respect to 
$\d \tilde N$, which by results of 
Dru{\c{t}}u-Osin-Sapir~\cite{DOS}
implies that $\pi_1(N)$ is hyperbolic relative
the fundamental groups of components of $\d N$.
It is unclear to me whether this proof also gives 
part (2) of Theorem~\ref{thm: intro-gr-theoretic-cor}. 

Part (14) of Theorem~\ref{thm: intro-gr-theoretic-cor}
follows from the following characterization 
of lattices in Lie groups that are relatively hyperbolic.

\begin{thm}\label{thm: rh lattices}
Let $G$ be a real Lie group with identity component $G_0$,
let $\Lambda$ be a lattice in $G$. 
Then $\Lambda$ is isomorphic to a non-elementary 
relatively hyperbolic group if and only if 
one of the following is true:\newline
\textup{(i)} $G_0$ is compact and
$G/G_0$ is isomorphic to a non-elementary relatively hyperbolic group;
\newline
\textup{(ii)}
$G/G_0$ is finite, and $G$ contains a compact normal subgroup $K$
such that $K\subset G_0$ and $G_0/K$ is 
a simple noncompact $\mathbb R$-rank one Lie group with trivial center. 
\end{thm}

\begin{rmk} The assertions (i), (ii) of
Theorem~\ref{thm: rh lattices} do not mention $\L$, thus
relative hyperbolicity of lattices in $G$ can be read off $G$.
However, it is instructive to see where $\L$ is ``hidden'': 
if $\L_0:=\L\cap G_0$, then one shows that
\begin{itemize}
\item
$\L_0$ is a lattice in $G_0$,
\item 
$G_0$ is compact if and only if $\L_0$ is finite,
\item
$\L/\L_0$ is a finite index subgroup of $G/G_0$,
\item
the restriction of the projection $G\to G/K$ to $\L$ has finite kernel.
\end{itemize}
\end{rmk}

None of the results of this paper
is truly hard to prove and they draw heavily on
various (often deep) works available in the literature. A few harder 
questions are below.

Let $M$ be a compact locally symmetric irreducible
$n$-manifold of $\sec\le 0$
and $S$ is a compact totally geodesic submanifold of codimension two
(where we assume $n>3$, else the answers to the questions below
are known to be ``yes''). 

\begin{quest}
Is $\pi_1(M\setminus S)$ quasi-isometricaly rigid?
\end{quest}

Theorem~\ref{thm: intro-mostow} suggests that the answer might be ``yes''.
Recall that for irreducible nonuniform lattices in semisimple 
real Lie groups with finite center quasi-isometry implies commensurability
(as proved by Schwartz, Eskin et al, see~\cite{Far-lat} for details).

\begin{quest}\label{quest: rf}
Is $\pi_1(M\setminus S)$ Hopfian? residually finite? linear? 
\end{quest}

There is currently no general method of establishing Hopf property
for relatively hyperbolic groups. Nevertheless, 
since the peripheral subgroups
of $\pi_1(M\setminus S)$ are well-understood and quite rigid,
the methods  of~\cite{Sel, DruSap-rips} may suffice to imply
that $\pi_1(M\setminus S)$ is Hopf.

For $(M,S)$ modelled on $({\bf X}_n, {\bf X}_{n-1})$
where ${\bf X}_n$ is the symmetric spaces corresponding to
$SO(n,2)$, Toledo~\cite{Tol-nrf} used Raghunathan's work~\cite{Rag}
to show that if $n\ge 4$ and $n$ is even, then $\pi_1(M\setminus S)$
is not residually finite. 

If $\pi_1(M\setminus S)$ is hyperbolic relative to the fundamental 
groups of the ends of $M\setminus S$, as happens when
$(M,S)$ is modelled on $(\rhn, \rhm )$ or $(\chn, \chm)$,
then Part (5) of Theorem~\ref{thm: intro-gr-theoretic-cor}, 
which is based on relatively hyperbolic 
Dehn Surgery theorem~\cite{Osi-periph} (cf.~\cite{GroMan}),
implies that $\pi_1(M\setminus S)$ is residually hyperbolic.
In particular, if $\pi_1(M\setminus S)$ is
{\it not} residually finite, then there exists a 
hyperbolic group that is not residually finite, which illustrates
the difficulty of Question~\ref{quest: rf}.

{\bf Structure of the paper.} 
In Section~\ref{sec: rh cyl coord} we write the real hyperbolic 
metric in cylindrical coordinates and also show that $\d N$ is a 
virtually trivial circle bundle over $S$.
Section~\ref{sec: curv comp} contains a curvature computation,
which is then used in Section~\ref{sec: rel hyp}
to prove parts (1)-(2) of Theorem~\ref{thm: intro-gr-theoretic-cor}.
Theorems~\ref{thm: intro-gr-theoretic-cor}, \ref{thm: intro-mostow}, 
\ref{thm: rh lattices} are proved in
Sections~\ref{sec: gr-theor-cor}, \ref{sec: mostow}, \ref{sec: lat rh}, 
respectively.
Appendix~\ref{app: rel hyp} contains definitions and basic results on 
relatively hyperbolic groups. A Morse-theoretic lemma describing topology
of the universal cover of $M\setminus S$ is proved in
Appendix~\ref{app: morse}. In Appendix~\ref{app: components-of-curv-tensor}
we review (and also correct) 
some curvature computation for a  
multiply-warped product metric that were 
worked out in~\cite{BW}.
Appendix~\ref{sec: another proof of rh} contains another proof of
part (1) of Theorem~\ref{thm: intro-gr-theoretic-cor}.

\section{Real hyperbolic space in cylindrical coordinates}
\label{sec: rh cyl coord}

We denote the real hyperbolic metric on $\rhn$ by ${\bf h}_n$.
Given a totally geodesic subspace $\rhm\subset\rhn$,
denote the distance to $\rhm$ by $r$, and write ${\bf h}_n$
as $dr^2+\r_r$,
where $\r_r$ is the induced metric on the $r$-tube $F(r)$
around $\rhm$.
It is well-known that $\r_r$ is the Riemannian product
of $\sinh^2(r) d\theta^2$ and $\cosh^2(r){\bf h}_{n-2}$
where $d\theta^2$ is the round metric on the unit circle, 
denoted $\S1$.
This fact is crucial for what follows hence
we shall outline a proof (which apparently is 
not recorded in the literature). 

The orthogonal projection 
$\pi\co\rhn\to\rhm$ is a fiber bundle whose fibers
are totally geodesic $2$-planes. Restricting $\pi$ 
to $F(r)$ gives a circle bundle $F(r)\to\rhm$.
This defines a splitting of the tangent bundle
$TF(r)=\cV(r)\oplus\cH(r)$, where $\cV(r)$ is tangent
to the circle fibers, and $\cH(r)$ is the orthogonal complement
of $\cV(r)$. Fix a point $z\in F(r)$. The metric on $\cV(r)$
at $z$ can be computed inside the totally geodesic $2$-plane
$\pi^{-1}(\pi(z))$, and using the polar coordinates
description of the hyperbolic $2$-plane
$dr^2+\sinh^2(r) d\theta^2$, we conclude that
the induced metric on $\cV(r)$ is $\sinh^2(r) d\theta^2$.
A key feature of the real hyperbolic case is that
at each point $z\in F(r)$ the subbundle
$\cH(r)$ is tangent to the codimension
one totally geodesic subspace spanned by $\rhm$ and $z$.
Thus the induced metric on $\cH(r)$ at $z$ can be computed
inside this codimension one subspace whose metric 
can be written as $dr^2+\cosh^2(r){\bf h}_n$, so that
the induced metric on $\cH(r)$ at $z$ is $\cosh^2(r){\bf h}_n$.
Therefore, the induced metric in $F(r)$ is a Riemannian
submersion metric with base $\cosh(r)\rhm$ and fiber
$\sinh(r)$-multiple of the unit circle. 
In this metric the circle fibers are closed geodesics,
because this can be checked in  the totally geodesic
$2$-plane $\pi^{-1}(\pi(z))$, where it is a tautology
that any curve is a geodesic in the induced metric on itself.
The $A$-tensor of the 
Riemannian submersion vanishes because $\cH(r)$ is an integrable
distribution being tangent to the foliation obtained by
intersecting $F(r)$ with the codimension one totally geodesic
subspaces containing $\rhm$. Thus the Riemannian submersion 
is locally a product, and hence globally a product
as $\rhm$ is simply-connected.
In summary, the induced metric on $F(r)$ is the Riemannian
product of $\cosh(r)\rhm$ and $\sinh(r)$-multiple 
of the unit circle. The isometry group of $F(r)$
is the product of the isometry groups of the factors, i.e.
$\mathrm{Iso}(\rhm )\times O(2)$.

Now if $\Gamma$ a torsion-free discrete isometry group of $\rhn$
stabilizing a codimension two totally geodesic subspace $\rhm$, 
then $\Gamma$ acts isometrically, freely, and properly
discontinuously on $F(r)$.
Hence $F(r)/\Gamma$ is a flat Euclidean circle bundle over
$\cosh(r)$-multiple of $\rhm/\Gamma$, where
``Euclidean'' means that the holonomy of the flat connection
lies in $O(2)$. 

We shall need the following fact that does not seem to be 
in the literature.

\begin{prop} \label{prop: flat bundle}
If $X$ is a connected manifold with 
finitely generated fundamental group,
then any flat Euclidean circle bundle over $X$  
becomes trivial, as a $O(2)$-circle bundle, after passing
to a finite Galois cover. In particular, the fundamental
group of the total space has a finite index normal subgroup
isomorphic to $\Z\times\pi_1(X)$. 
\end{prop}

\begin{proof} 
If  $\tilde X$ denotes the
universal cover of $X$, then any flat Euclidean 
circle bundle over $X$ is a  $\pi_1(X)$-quotient of
$\tilde X\times S^1$, where $\pi_1(X)$ acts by covering automorphisms on
the first factor, and via the holonomy homomorphism on the second factor.
Passing to the orientation cover of the bundle cover, 
we can assume that the bundle is orientable, while the fundamental
group of the base is still finitely generated. Orientable flat Euclidean 
circle bundles are classified by holonomy $\pi_1(X)\to SO(2)$,
so let $\varphi$ be the holonomy homomorphism of the flat structure.
Since $SO(2)$ is abelian, $\varphi$
factors through $\bar\varphi\co H_1(X)\to SO(2)$. 
Write $H_1(X)=\Z^k\oplus T$
where $T$ is the (finite) torsion subgroup, and deform the $\Z^k$ 
summand to the trivial subgroup inside $SO(2)$ without changing
$\bar\varphi_{|T}$.
(To do so choose a generating set $s_1,\dots s_k$ for $\Z^k$, 
fix arbitrary paths $p_i(t)$ in $SO(2)$ 
from $\bar\varphi(s_i)$ to $1$, and define $\bar\varphi_t(s_i)=p_i(t)$
and $\bar\varphi_{t|T}=\bar\varphi_{|T}$).
The endpoint is $\bar\varphi_1\co\pi_1(X)\to SO(2)$
whose image $\bar\varphi_1(T)$ is finite. Precomposing with
the abelianization $\pi_1(X)\to H_1(X)$ yields a path
in the representation variety $\mathrm{Hom}(\pi_1(X), SO(2))$
from $\varphi$ to a homomorphism with image $\bar\varphi (T)$.
By the covering homotopy theorem, the corresponding flat $SO(2)$-bundles
have isomorphic underlying $SO(2)$-bundles. Clearly,
if a bundle corresponds to a homomorphism $\pi_1(X)\to SO(2)$ with
finite image, then it becomes trivial in the cover that corresponds
to the kernel of the homomorphism.
\end{proof}

\begin{rmk}
A quicker but less elementary way to see that the bundle $\d N\to S$ 
is virtually trivial is to note that $\d N$ is a totally geodesic 
submanifold in the nonpositively curved
Heintze-Schroeder metric on $N$, so that the fundamental group
of each component of $\d N$ is $CAT(0)$, which implies that 
the centralizers virtually split~\cite[Theorem 1.1(iv), page 439]{BH}.
\end{rmk}

\section{Curvature computation}
\label{sec: curv comp}

The curvature computations done in this section are straightforward
and they could have been omitted had this paper been written
for differential geometers. Since we expect more diverse
readership, full details are given. Also instead of using ad hoc
arguments, we find it more illuminating to rely on general 
curvature formulas developed in~\cite{BW} and reviewed in 
Appendix~\ref{app: components-of-curv-tensor}.
This section could serve as a gentle introduction to
delicate curvature computations in~\cite{Bel-ch-warp},
which use Appendix~\ref{app: components-of-curv-tensor}
in essential way.

Given positive functions $v, h$, we write 
the Riemannian product metric of $v(r)\S1$ and $h(r)\rhm$
as $\l_{r,v, h}=v^2(r)d\theta^2+h^2(r){\bf h}_{n-2}$.
Consider the metric $\l_{v,h}=dr^2+\l_{r,v, h}$ on $I\times F$,
where $I$ is an open interval, and
$F$ is an underlying smooth manifold of ${\S1}\times \rhm$.
Thus if $v(r)=\sinh(r)$, $h(r)=\cosh(r)$
and $I\subseteq (0,\infty)$, then $\l_{v,h}$ coincides
with the real hyperbolic metric ${\bf h}_n$.
For brevity
we sometimes suppress $v,h$ and label tensors associated with
$\l_{v,h}$, $\l_{v,h,r}$ by $\l$, $\l_r$, respectively, and 
also denote by $K$ the sectional curvature of $\l$. 
%
%
%

Next we define a local orthonormal
frame on $I\times F$ in which the curvature of 
$\l_{v,h}$ will be computed.
Denote $\frac{\d}{\d r}$ by $\ddr$, and 
$\frac{\d}{\d\theta}$ by $X_1$.
Fix $z\in I\times F$ and let $w\in\rhm$ be the image
of $z$ under the projection to the last coordinate
$p\co I\times F\to\rhm$.
Consider an arbitrary orthonormal frame $\{\check X_i\}$, 
with $1<i<n$, defined on a neighborhood of $w$ in $\rhm$ such that 
$[\check X_i,\check X_j]$ vanishes at $w$ for all $i,j$.
(By a standard argument any orthonormal basis 
in $T_w\rhm$ can be extended to some $\{\check X_i\}$ as above). 
Let $X_i$ be the vector field on $I\times F$
obtained by lifting $\check X_i$ via the coordinate inclusions 
$\rhm\to I\times F$.
Then $\ddr, X_1,\dots , X_{n-1}$ is an orthogonal frame near $z$ such that
\begin{itemize}
\item [\rm(1)] $\langle X_1,X_1\rangle_\l=v^2$, 
and $\langle X_i, X_i\rangle_\l =h^2$ for $i>1$. 
\item [\rm(2)] $[X_i, X_j]=0=[X_i,\ddr]$ at $z$ for all $i,j$,
\end{itemize}
where all brackets vanish at $z$ as the manifold is smoothly a product.
The corresponding orthonormal frame $\ddr$, 
$Y_1=\frac{1}{v}X_1$, $Y_i=\frac{1}{h}X_i$, $i>1$ 
has the following properties:
 \begin{itemize}
\item[\rm(i)] $[Y_i, Y_j]=\frac{1}{h^2}[X_i, X_j]=0$ for $i,j>1$,
\item[\rm(ii)] $[Y_i, Y_1]=\frac{1}{hv}[X_i, X_1]=0$,
%
%
\end{itemize}
where the first equalities in (i), (ii) hold because 
any function of $r$ has zero derivative in the direction of $X_i$. 
%

Both $v(r)\S1$ and $h(r)\rhm$ have constant sectional curvature
so computing the curvature tensor of $R_{\l_r}$
is straightforward, e.g. it follows
from~\cite[Corollary V.2.3]{KN} and~\cite[Theorem 9.28]{Bes} 
that up to symmetries of the curvature tensor
the only nonzero components of $R_{\l_r}$ are the sectional
curvatures of $Y_iY_j$-planes with $1<i<j<n$,
which is $-h^{-2}$.

Combining this with the fact that all brackets $[Y_i, Y_j]$ vanish,
we see from Appendix~\ref{app: components-of-curv-tensor} that
up to symmetries of the curvature tensor
the only nonzero components of the curvature tensor $R_\l$ are
\begin{eqnarray}
&\label{form: rh, cood pl one}
\ \ \ \ \sec_\l(Y_i,Y_1)=-\frac{h^\prime v^\prime}{hv},
\ \ \ \ \  \sec_\l (Y_i,\ddr)=-\frac{h^{\prime\prime}}{h},
\ \ \ \ \  \sec_\l (Y_1,\ddr)=-\frac{v^{\prime\prime}}{v},\\
& \label{form: rh, cood pl two}
\sec_\l (Y_i,Y_j)=-\frac{1}{h^2}-\left(\frac{h^\prime}{h}\right)^2
\ \text{where}\ 1<i<j<n.
\smallskip
\end{eqnarray}
Hence the bivectors $Y_k\wedge Y_l$, $Y_i\wedge \ddr$ with $k<l$
diagonalize the curvature operator, and hence global maxima and minima
of $\sec_\l$ are attained on the coordinate planes $Y_kY_l$, 
$Y_i\ddr$, see e.g.~\cite[Chapter 3, Section 1]{Pet}.
Thus to achieve $\sec_\l<0$
we just need to choose $h$, $v$ increasing and satisfying
$h^{\prime\prime}>0$, $v^{\prime\prime}>0$. 

Let $M$, $S$ be as in Theorem~\ref{thm: intro-gr-theoretic-cor}.
We are going to modify the incomplete real hyperbolic 
metric on the ends of $M\setminus S$ near $S$.
Suppose $\e$ is less than the half of the
normal injectivity radius of $S$ in $M$, and let $I=(-\infty,\e)$. 
Let $\r\gg 1$ be a parameter. We let $h(r)=v(r)=e^r$ 
on $(-\infty, -\r]$, let 
$h(r)=\cosh(r)$, $v(r)=\sinh(r)$ on $[\e, \infty)$, and then
interpolate on $[-\r, \e]$ to ensure that $h$, $v$
are smooth and $h^\prime$, $v^\prime$,
$h^{\prime\prime}$, $v^{\prime\prime}$ are positive on $[-\r, \e]$, 
which is clearly possible for
sufficiently large $\r$ (e.g. choose $\r$ so that the tangent lines
to $e^r$ at $r=-\r$ and to $\cosh(r)$, $\sinh(r)$ at $r=\e$ 
intersect on $(-\r,\e)$). Consider the corresponding metric $\l=\l_{v,h}$. 
By (\ref{form: rh, cood pl one})--(\ref{form: rh, cood pl two})
one checks that $\sec_\l\le -1$ if $r\le-\r$, and
of course $\sec_\l=-1$ if $r\ge \e$. So compactness
of $[-\r, \e]$ implies that $\sec_\l$
is bounded above by a negative constant
depending only on $\r, \e$. The metric $\l$ descends to 
a metric on the ends of $M\setminus S$ near $S$;
the metric will be also denoted $\l=\l_{v,h}$.

\begin{rmk} 
Every end of $M\setminus S$ with metric $\l_{v,h}$
admits a Riemannian submersion
onto $(-\infty, 0]$, and we refer to its fibers
as {\it cross-sections}. In turn, every cross-section 
is the total space of a Riemannian submersion with fiber $v\S1$ 
and base $h B$, where $B$ is a component of $S$.
The volume of the cross-section can be computed by Fubini's theorem 
for Riemannian submersions~\cite[Corollary 5.7]{Sak},
and since $v\S1$ has volume $2\pi v$, the volume of the cross-section
is $2\pi v h^{2n-2}\vol(B)$. By the same result, the volume of
of the portion of the end corresponding to $[r, 0]$
is $2\pi\vol(B)\int_{r}^{0} v h^{2n-2}dr$ and the volume of the end
is obtained by taking $r$ to $-\infty$. It follows that the ends
have finite volume provided $h$ is bounded and 
$\int_{-\infty}^{r_0} v<\infty$, which holds for $\l_{v, h}$.
Thus if the ends of $M\setminus S$ that approach $S$
are given the metric $\l_{v,h}$, then $M\setminus S$ 
has finite volume.
\end{rmk}

\begin{rmk}\label{rmk: HSF metrics}
(1) By choosing $v, h$ equal to small constants on $(-\infty,\r]$ 
one recovers the Heintze-Schroeder's 
metric~\cite{Sch} of $\sec\le 0$ that is a product metric
at the ends that approach $S$. \newline
(2) By making $v(r)=e^r$
and $h(r)=e^r+\tau$ on $(-\infty,\r]$, 
for some small positive $\tau=\tau(\r, \e)$ 
one recovers Fujiwara's result~\cite{Fuj} that $M\setminus S$ admits
a complete finite volume metric of sectional curvature within $[-1,0)$.
\end{rmk}

\begin{rmk} 
\label{rmk: pinched curv in dim 2,3}
If $n\le 3$, then  $\sec_\l$ is bounded between
two negative constants because $Y_iY_j$ with 
$1<i<j<n$ is the only coordinate plane on which there is
no lower curvature bound, and $1<i<j<n$ cannot occur
when $n\le 3$.
\end{rmk}

\section{Proving relative hyperbolicity}
\label{sec: rel hyp}

We refer to Appendix~\ref{app: rel hyp} for basic information
on relatively hyperbolic groups. 
As is briefly explained in~\cite[Section 6]{Bow-rel},
Gromov's definition of relative hyperbolicity 
(stated somewhat informally in~\cite[Definition 8.6.A]{Gro-hgr})
is equivalent to Definition~\ref{defn: rel hyp}, which is
adopted in the present paper. 
Proposition~\ref{prop: gromov defn of rh} below is essentially
the assertion 
that Gromov's definition implies Definition~\ref{defn: rel hyp}
stated for group actions on negatively curved Hadamard manifolds.
Unfortunately, we could not find the precise statement we need
in the literature, so we supply a proof.

\begin{prop} \label{prop: gromov defn of rh}
Let $G$ be a group acting 
properly discontinuously and isometrically on a complete simply-connected 
Riemannian $n$-manifold $X$ with $\sec(X)\le -1$ and $n\ge 2$.
Suppose that $X$ contains a $G$-invariant family of horoballs
$\mathcal B$ such that \newline
\textup{(1)} for some $r>0$ the distance between any
two distinct horoballs in $\mathcal B$ is $\ge r$,\newline 
\textup{(2)}
$G$ acts cocompactly on the space $X_0$ obtained by removing from $X$ 
the interiors of the horoballs in $\mathcal B$.\newline
Then $G$ is a non-elementary relatively hyperbolic, 
relative to the set of stabilizers of the horoballs in $\mathcal B$.
The relatively hyperbolic boundary of $G$ is $\d X$,
which is homeomorphic to $(n-1)$-sphere.
\end{prop}

\begin{proof}
Let $\Pi$ be the set of centers of horoballs in $\mathcal B$.
First, we show that any point of $\Pi$ is bounded parabolic.
In fact, if $z$ is the center of a horoball $B\in\mathcal B$, 
then the stabilizer $\mathrm{Stab}_G(B)$ of $B$ in $G$ is 
a parabolic subgroup fixing $z$ and acting cocompactly on the horosphere $\d B$.
(Indeed, $\mathrm{Stab}_G(B)$ fixes $z$,
and hence $\mathrm{Stab}_G(B)$ is an elementary subgroup
that acts cocompactly and properly discontinuously on 
the horosphere $\d B$, which is noncompact.
In particular, $\mathrm{Stab}_G(B)$ is infinite. Hence either  
$\mathrm{Stab}_G(B)$ is a parabolic subgroup, or else
$\mathrm{Stab}_G(B)$ stabilizes a geodesic in $X$ 
emanating from $z$;
the latter alternative is impossible because then 
$\mathrm{Stab}_G(B)$ would have to 
fix the point of $X$ where the geodesic 
intersects $\d B$.)

Second, any point of $y\in\d X\setminus\Pi$ is conical limit
point. Indeed, any ray emanating from $y$ intersects $X_0$ in a unbounded
set (else the ray would eventually stay in one horoball, so $y$
would be its center contradicting $y\notin\Pi$). 
Take $x_i\in X_0$ that lie on the ray and tend to
$y$. Since $X_0/G$ is compact, it has finite
diameter $D$, so there is an infinite $G$-orbit $g_i(x)$ with
$\mathrm{dist}(x_i, g_i(x))\le D$, hence $y$ is conical.

By Definition~\ref{defn: rel hyp}
it follows that $G$ is hyperbolic relative to its 
maximal parabolic subgroups, and the relatively hyperbolic boundary equals
to $\d X$, which is an $(n-1)$-sphere, in particular, 
$\d X$ contains more than two points, so $G$ is non-elementary.

Next observe that maximal parabolic subgroups of 
$G$ are precisely the subgroups $\mathrm{Stab}_G(z)$, $z\in\Pi$.
Indeed,
no conical limit point can be fixed by a parabolic 
subgroup~\cite[Proposition 3.2]{Bow-conv},
so if $z$ is the fixed point of a maximal parabolic subgroup, 
then $z\in\Pi$ and the maximal parabolic subgroup lies in
$\mathrm{Stab}_G(z)$.
Conversely, any point $z\in\Pi$ is fixed by the 
parabolic subgroup $\mathrm{Stab}_G(z)$,
which is contained in a maximal parabolic subgroup
whose unique fixed point must be $z$.

Finally, note that if the horoball $B\in\mathcal B$ 
is centered at $z$, then
$\mathrm{Stab}_G(B)$ equals to the stabilizer 
$\mathrm{Stab}_G(z)$ of $z$ in $G$: the inclusion
$\mathrm{Stab}_G(B)\subseteq\mathrm{Stab}_G(z)$ is obvious, and
if $g(z)=z$, then $g(B)$ and $B$ are concentric horoballs
in $\mathcal B$, hence $g(B)=B$ by the assumption (1).
\end{proof}

A neighborhood $E$ of an end in a complete Riemannian manifold
is called a {\it cusp neighborhood} if 
$E$ admits a Riemannian submersion onto $[0,\infty)$, and
there exists a constant $K$ such that the ``holonomy'' diffeomorphism
$h_t$ from the fiber over $\{0\}$ to the fiber over $\{t\}$ 
is $K$-Lipschitz for each $t$.

To make sense of the above definition 
recall that Riemannian submersions of complete manifolds
are smooth fiber bundles, so fibers of $E\to [0,\infty)$
are closed smooth submanifolds, and we denote the fiber
over $\{t\}$ by $F_t$. 
Through every point of $E$ there
exists a unique horizontal geodesics ray
that start at $F_0$; the ray intersects each $F_t$ orthogonally 
and projects isometrically onto $[0,\infty)$.
Pushing along horizontal geodesics ray defines
the diffeomorphism $h_t\co F_0\to F_t$.

\begin{thm} \label{thm: rel hyp}
For $n\ge 2$, $k\ge 1$, 
suppose that $V$ is a complete Riemannian $n$-manifold with
$k$ ends, which have disjoint 
cusp neighborhood $E_1,\dots, E_k$. 
Let $q\co \tilde V\to V$ denote the universal cover of $V$.
Then\newline
\textup{(1)} If $\sec(V)\le 0$,
then for each $i$, the inclusion $E_i\to V$ 
is $\pi_1$-injective, and every component
of $q^{-1}(E_i)$ is a horoball.\newline
\textup{(2)} If  $\sec(V)$ is bounded above by a negative constant, 
and each of the submersions $E_i\to [0,\infty)$ has compact fibers,
then $\pi_1(V)$ is non-elementary relatively hyperbolic, 
relative to the conjugates
of $\pi_1(E_1), \dots \pi_1(E_k)$ in $\pi_1(V)$, and the relatively 
hyperbolic boundary is the ideal boundary of $\tilde V$,
which is homeomorphic to $S^{n-1}$.
\end{thm}

\begin{proof} 
Fix an end of $V$, and let $E\in\{E_1,\dots , E_k\}$ 
be the corresponding cusp neighborhood of the end.
Fix an arbitrary  component $C$ of $q^{-1}(E)$, 
so that $q$ restricts to a covering $C\to E$. 
We next show that $C$ is a horoball.
The Riemannian submersion $E\to [0,\infty)$ lifts to a Riemannian
submersion $C\to [0,\infty)$, whose fibers are denoted $\tilde F_t$.
The diffeomorphisms $h_t$ lifts to a diffeomorphism
$\tilde h_t\co\tilde F_0\to \tilde F$ defined by pushing along
horizontal geodesic rays of $C\to [0,\infty)$.
Since each $\tilde h_t$ is clearly $K$-Lipschitz,
any two horizontal geodesic rays of
the submersion $C\to [0,\infty)$ are asymptotic in $\tilde V$
(because the Hausdorff distance between the rays is bounded above
by $K$ times the distance between their initial points
in $\tilde F_0$).
So the horizontal geodesic rays define a point at infinity of $\tilde V$. 
The horospheres in $\tilde V$ centered at the point
are orthogonal to the horizontal rays in $C$ and hence 
by dimension reasons their tangent spaces coincide with 
tangent spaces of fibers of $C\to [0,\infty)$.
Since a foliation is determined by its tangent subbundle,
the foliations of $C$ by horospheres and by the fibers $F_t$ coincide.
As fibers and horospheres are codimension one closed submanifolds 
of $\tilde V$, every fiber is a closed and open subset of a horosphere,
so each $\tilde F_t$ equals to a horosphere.
Thus $C$ is the union of the horospheres that intersect a
fixed horizontal geodesic ray in $C$, i.e. $C$ is a horoball.
Since horoballs are contractible, the inclusion
$E\to V$ is $\pi_1$-injective, which proves (1).

Let $G:=\pi_1(V)$ and $X:=\tilde V$, and let
$\mathcal B$ be the set of horoballs in $X$ that
are components of $q^{-1}(\cup_i E_i)$.
To prove (2) we are going to check that the $G$-action on $X$
satisfies the assumptions of Proposition~\ref{prop: gromov defn of rh}.
Let $L:=V\setminus\cup_i E_i$, i.e. 
$L$ is the $G$-quotient of the space obtained from $X$
by removing all horoballs of $\mathcal B$.
By assumptions the closure of $L$ in $V$
is a compact manifold with boundary, and distinct 
horoballs are $r$-separated e.g. by setting $3r$ equal to
the normal injectivity radius of $\d L$.
By Proposition~\ref{prop: gromov defn of rh}
$G$ is a non-elementary relatively hyperbolic, 
relative to the set of stabilizers of the horoballs in $\mathcal B$.
Finally, the stabilizers of horoballs in $\mathcal B$
are precisely the conjugates of the subgroups
$\pi_1(E_1),\dots ,\pi_1(E_k)$, which completes the proof.
\end{proof}

\begin{proof}[Proof of \textup{(1)-(2)} 
in Theorem~\ref{thm: intro-gr-theoretic-cor}]
By Section~\ref{sec: curv comp}, we
equip $M\setminus S$ with a complete metric of sectional curvature
bounded above by a negative constant. Each end of $M\setminus S$
corresponding to a cusp of $M$ has a neighborhood with warped product metric
$dr^2+e^{2r}df^2$ where $df^2$ is a flat metric on a cusp cross-section.
Each end of $M\setminus S$ that approaches $S$ has a neighborhood
with metric
$dr^2+v^2 d\theta^2+h^2{\bf h}_{n-2}$.
In either case the $r$-coordinate projection is a Riemannian submersion
with compact fibers, and $h_t$ is $1$-Lipschitz as $e^r, v, h$ are 
increasing, so these are cusp neighborhoods, and Theorem~\ref{thm: rel hyp}
applies.
\end{proof}

\begin{rmk}
The conclusion that the relatively hyperbolic boundary
is a sphere is of interest in itself. Previously known
relatively hyperbolic groups with sphere boundaries are
the fundamental groups of complete
Riemannian orbifolds of finite volume and pinched negative curvature,
or closed aspherical piecewise hyperbolic locally $CAT(-1)$
manifolds whose links are topological spheres~\cite{DavJan}; 
these manifolds 
can be produced via the strict hyperbolization of polyhedra~\cite{ChaDav}.
\end{rmk}

\section{Proof of Theorem~\ref{thm: intro-gr-theoretic-cor}}
\label{sec: gr-theor-cor}

The parts (1)-(2) are proved in Section~\ref{sec: rel hyp}.

(3) immediately follows from (1) and a result of the 
author~\cite[Theorem 1.3]{Bel-HP},
where it was proved that if $N$ is a compact aspherical manifold
with incompressible boundary and $\pi_1(N)$ is hyperbolic relative to
fundamental groups of components of $\d N$ and $\dim(N)>2$, then $\pi_1(N)$
does not split over elementary subgroup of the relatively hyperbolic group
$\pi_1(N)$.

(4) A group is called {\it co-Hopf} if any
injective endomorphism is surjective. The co-Hopf property
for $\pi_1(N)$ follows from (1) and (3) together with a result of
Dru{\c{t}}u-Sapir~\cite{DruSap-rips} (see~\cite[Theorem 1.3]{Bel-HP}
for details). Alternatively, by~\cite[Theorem 1.5]{Bel-HP}, 
the co-Hopf property for $\pi_1(N)$ follows from (1)
together with results of Mineyev-Yaman~\cite{MinYam} who showed
that hyperbolicity of $\pi_1(N)$ relative to the
boundary implies positivity of the relative simplicial volume 
$||N, \d N||$.

(5) Given a class of groups $\mathcal C$,
a group $\Gamma$ is called {\it fully residually} $\mathcal C$
if any finite subset of $\Gamma$ can be mapped injectively
by a homomorphism of $\Gamma$ onto a group in $\mathcal C$.
Osin~\cite{Osi-periph} proved that if all peripheral subgroups
of a non-elementary relatively hyperbolic group $G$ are 
fully residually hyperbolic, then $G$ is fully residually 
non-elementary hyperbolic.
(This result can also be deduced from~\cite{GroMan}
provided $G$ is torsion-free and finitely generated, 
which is true for $G=\pi_1(N)$).
Now (5) follows from (1) together with 
the fact that the peripheral subgroups of $\pi_1(N)$ are 
fully residually finite, and hence fully residually hyperbolic
because finite groups are hyperbolic.

%

To see that peripheral subgroups of
$\pi_1(N)$ are fully residually finite recall that 
the components of $\d N$ corresponding to the cusps of $M$
have finitely generated virtually abelian fundamental groups, and
of course, finitely generated abelian groups are residually finite.
By Proposition~\ref{prop: flat bundle}
the fundamental group of any other component of $\d N$
contains a finite index subgroup isomorphic to $\Z\times\pi_1(B)$
where $B$ is a component of $S$. The groups $\Z$, $\pi_1(B)$ 
are linear and finitely generated, hence so is their product, 
which implies residual finiteness of $\Z\times\pi_1(B)$
by a classical result of Mal'tsev. 
Thus all components of $\d N$ have residually finite fundamental groups.
Finally, if the class $\mathcal C$ closed under
finite direct products, then it is immediate that any 
residually $\mathcal C$ group is fully residually $\mathcal C$,
and we conclude that $\pi_1(N)$ is fully residually finite.

(6) A group satisfies 
the {\it Strong Tits Alternative} if any subgroup either contains
a nonabelian free group or is virtually abelian.
Tukia~\cite{Tuk} (cf.~\cite[Secton 8.2.F]{Gro-hgr}) 
proved the following Tits Alternative for relatively 
hyperbolic groups: a subgroup that does not contain a non-abelian
free subgroup is either finite, or virtually-$\Z$, or lies in
a peripheral subgroup. Thus it suffices to prove
the strong Tits alternative for the peripheral subgroups, which
are extensions of non-elementary hyperbolic groups by $\Z$ 
(the fundamental group of a circle bundle over a closed 
hyperbolic manifold is such by the homotopy exact 
sequence of the bundle). 
Suppose $Q$ is hyperbolic and $p\co G\to Q$ is a quotient with 
infinite cyclic kernel. 
If $H\le G$ does not contain  $\Z \ast \Z$, then neither does
$p(H)$, so $p(H)$ is trivial or virtually-$\Z$, which immediately
implies that $H$ is virtually abelian of rank $\le 2$.

(7)
According to Rebecchi~\cite{Reb} 
a relatively hyperbolic group is biautomatic 
provided its peripheral subgroups are biautomatic. 
Virtually abelian groups are biautomatic~\cite{ECHLPT},
and so are virtually central extensions of hyperbolic groups~\cite{NeuRee},
so (1) implies that $G$ is biautomatic.

(8) Lemma~\ref{lem: morse}, proved in Appendix B, 
implies that the kernel $K$ of the surjection 
$\pi_1(M\setminus S)\to \pi_1(M)$ is a (countably generated) free group. 
If $H$ is a subgroup of $\pi_1(M\setminus S)$ with 
Kazhdan property (T), then it lies in the kernel
because any action of $K$ on a real hyperbolic space has a 
fixed point~\cite{HarVal}, while $\pi_1(M)$ acts freely. 
Thus the group $H$ is free, and hence trivial, as nontrivial
free groups do not have property (T).

\begin{rmk}
Similarly, one could show that $\pi_1(N)$ contains no nontrivial
simple subgroup for if $K$ is a simple subgroup of $\pi_1(N)$,
then either $H\subset K$ or $H\cap K=\{1\}$. In the former
case $H$ is free hence trivial, as nontrivial free groups are not simple. 
In the latter case $H$ embeds into $\pi_1(M)$, which is residually
finite, so that $H$ is residually finite and hence not simple.
\end{rmk}

(9) 
Arguing by contradiction, suppose that $K$ is a compact K\"ahler manifold 
with $\pi_1(K)$ isomorphic to $\pi_1(M\setminus S)$. 
Since $S$ has codimension two,
the inclusion $M\setminus S$
induces a surjection $\pi_1(M\setminus S)\to\pi_1(M)$. 
Precomposing with the isomorphism
$\pi_1(K)\cong \pi_1(M\setminus S)$ we get a surjection 
$\pi_1(K)\to\pi_1(M)$, which is homotopic to a non-constant
harmonic map $K\to M$~\cite{Lab}.
It was independently proved
in~\cite[Theorem 7.1]{CarTol} and~\cite[Theorem 3]{JosYau} 
that the harmonic
map $K\to M$ factors through a closed geodesic or a compact 
Riemann surface $\Sigma$. Since $\pi_1(M)$
is not virtually abelian, the map cannot factor through a geodesic, 
and $\Sigma$ must have negative Euler characteristic.

Fix a component $C$ of $S$, and let $B$ be the corresponding
component of $\d N$.
By Proposition~\ref{prop: flat bundle}
$B$ is a virtually trivial circle bundle over $C$.
So there exists a finite cover $\bar C\to C$
and a $\pi_1$-injective map $\bar C\to B$
such that  
$\bar C\to B$ followed by the bundle projection $B\to C$
is homotopic to $\bar C\to C$.
The inclusion $B\to N\subset M\setminus S$ is 
$\pi_1$-injective, and furthermore, the inclusion $B\to \d N\to M$ 
is homotopic to the  bundle projection $B\to C$ followed by
the inclusion $C\to M$.
Thus the covering $\bar C\to C\subset M$ homotopy factors through
the inclusion $B\to M$. It follows that 
$\pi_1(K)$ contains a subgroup isomorphic to 
$\pi_1(\bar C)$, and the above harmonic map
maps this subgroup isomorphically 
onto a finite index subgroup of $\pi_1(C)\le\pi_1(M)$. 
Since the harmonic map factors through $\Sigma$,
the group $\pi_1(\bar C)$
injects into $\pi_1(\Sigma)$.  
So $\pi_1(\bar C)$ has cohomological dimension $\le 2$. 
Since $\bar C$ is aspherical, $\bar C$ has dimension $\le 2$, 
and hence $\dim(S)\le 2$.

Nontrivial free groups are not K\"ahler~\cite[Example 1.19]{ABCKT}
so $\dim(S)>0$. 
If $S$ is $2$-dimensional, then the above $\pi_1$-injective
map $\bar C\to\Sigma$
is homotopic to a finite cover, so its image 
of $\pi_1(\Sigma)\to\pi_1(M)$ is commensurable to $\pi_1(C)\le\pi_1(M)$,
and in particular, $\pi_1(\Sigma)\to\pi_1(M)$ is not
onto, which is not the case.
If $S$ is one dimensional, then $M\setminus S$ is a real hyperbolic
$3$-manifold with finitely many closed geodesics removed.
Then it is a standard observation that 
$M\setminus S$ can be given negatively pinched metric 
e.g. by Remark~\ref{rmk: pinched curv in dim 2,3}, 
hence $N$ is irreducible and atoroidal, so 
by Thurston's hyperbolization theorem $M\setminus S$ admits 
a finite volume real hyperbolic metric.
If the group $\pi_1(M\setminus S)$ is K\"ahler, then we can repeat the
above argument for a harmonic map from $K$ to this real hyperbolic 
$3$-manifold,
that induces a $\pi_1$-isomorphism. We then conclude that 
the group $\pi_1(K)\cong\pi_1(M\setminus S)$ injects into $\pi_1(\Sigma)$. 
But $\pi_1(M\setminus S)$ contains 
$\mathbb Z\oplus\mathbb Z$ as a cusp subgroup, which contradicts
the fact that $\Sigma$ is has negative Euler characteristic.

(10) 
Osin proved in~\cite{Osi-asydim} 
that a relatively hyperbolic group has finite 
asymptotic dimension, provided all peripheral subgroups 
have finite asymptotic dimension.
It is known (see e.g.~\cite[Corollary 60]{BelDra})
that asymptotic dimension of finitely generated groups
is invariant under commensurability. 
Furthermore, the asymptotic dimension of 
an extension of finitely generated groups
is bounded above by the sum of the asymptotic dimensions 
of the kernel and the quotient~\cite[Theorem 7]{BelDra-ext},
Since $\Z$ has asymptotic dimension $1$,
any polycyclic finitely generated group
has finite asymptotic dimension, and in particular, this
applies to finitely generated nilpotent groups.
It remains to deal with extensions whose kernel is $\Z$, 
and the quotient is word-hyperbolic. 
For word-hyperbolic groups finiteness of 
asymptotic dimension is due to Gromov (see~\cite{Roe}). 

(11) Dru{\c{t}}u-Sapir~\cite{DruSap-RD} 
proved that a relatively hyperbolic group has rapid decay
property provided so do all peripheral subgroups.
Rapid decay property for groups of polynomial growth
(and in particular virtually abelian groups)
was established by Jolissaint~\cite[Proposition 2.1.9]{Jol}   
who also proved that the rapid decay property is invariant
under commensurability~\cite[Section 2.1]{Jol}.
So it remains to prove rapid decay property for central extensions
with kernel $\Z$ and hyperbolic quotient which was done
by Noskov~\cite{Nos}.

$\mathrm{(12)}$ 
Lafforgue~\cite[Corollary 0.0.4]{Laf-BC} 
proved Baum-Connes conjecture for groups with 
rapid decay property that are isomorphic to the fundamental 
groups of $A$-regular nonpositively curved manifolds.
A Riemannian metric is called {\it $A$-regular} if for every $k\ge 0$,
each component of the $k$th covariant derivative of the curvature tensor
is a bounded function on the manifold. 
By obvious compactness considerations,
any metric that is a product near infinity is $A$-regular, 
hence the Heintze-Schroeder's metric~\cite{Sch} on $M\setminus S$
is $A$-regular, so the result follows from (11).

\begin{rmk} Farrell-Jones showed that
the fundamental group of a 
complete $A$-regular manifold 
of $\sec\le 0$ satisfies the 
Borel's Conjecture~\cite[Proposition 0.10]{FJ-A-reg},
which says that the $L$ and $K$-theory assembly maps are isomorphisms.  
In particular, if $n\ge 5$, then any homotopy equivalence
of compact manifolds $L\to N$ that restricts to a homeomorphism 
$\d L\to\d N$ is homotopic rel boundary to a 
homeomorphism~\cite[Addendum 0.5]{FJ-A-reg}.
\end{rmk}

$\mathrm{(13)}$
In a discrete isometry group of a negatively pinched Hadamard manifold
the centralizer of any infinite order element is virtually nilpotent.
Indeed, an infinite order element is either hyperbolic or parabolic.
if the element is hyperbolic, then it
has exactly two fixed points at infinity which then must be stabilized
by the centralizer. Hence an index two subgroup of the centralizer 
fixes the two points, and hence acts properly discontinuously on the
geodesic joining the points. This shows that the centralizer is virtually 
cyclic. If the infinite order element is parabolic, then it fixes exactly
one point at infinity, and this point then must be fixed by the 
centralizer. According to Bowditch~\cite{Bow-par} the centralizer is
finitely generated and hence Margulis's Lemma shows it is also
virtually nilpotent. If $n>3$, then $\pi_1(N)$ has a peripheral subgroup
that is an extension with infinite cyclic kernel and non-elementary
hyperbolic quotient. Since $\mathrm{Aut}(\Z)$ has order two, 
its index two subgroup centralizes the infinite cyclic kernel,
yet the group is not virtually nilpotent, because it surjects
onto a non-elementary hyperbolic group. 

$\mathrm{(14)}$ 
Suppose that $\pi_1(N)$ embeds as a lattice 
$\L$ into a real Lie group $G$ with identity
component $G_0$, so Theorem~\ref{thm: rh lattices} applies.
Case (ii) is impossible, indeed $\L_0:=\L\cap G_0$ 
is isomorphic to a finite index subgroup of $\pi_1(N)$,
so $\L_0$ is torsion-free, hence it projects isomorphically
onto a lattice in $G_0/K$. So $\L_0$ embeds as a discrete subgroup
into the isometry group of a negatively curved symmetric space,
and we get a contradiction because 
by the proof of (13) $\L_0$ is not isomorphic to the 
fundamental group of a complete negatively pinched manifold.
Thus we are in Case (i), and the claim follows from the fact
that $\pi_1(N)$ is torsion free and the remarks made
after the statement of Theorem~\ref{thm: rh lattices},
which are justified in the proof of Theorem~\ref{thm: rh lattices}.

\begin{rmk} Of course,
if $\dim(M)\le 3$, then $M\setminus S$ carries
a complete finite volume real hyperbolic metric: 
the $2$-dimensional case is classical, and the $3$-dimensional case
follows from Thurston's hyperbolization theorem. Thus the
assumption $n>3$ is needed in (13), (14).
\end{rmk}

\section{Mostow-type rigidity for hyperplane complements}
\label{sec: mostow}

Theorem~\ref{thm: intro-mostow}
follows from Mostow-Prasad rigidity once it is shown that
any homotopy equivalence $M_1\setminus S_1\to M_2\setminus S_2$
takes ends of  $M_1\setminus S_1$ to ends of $M_2\setminus S_2$,
which is ensured by part (1)  of Theorem~\ref{thm: intro-gr-theoretic-cor}
together with an observation that
the conjugacy classes of peripheral subgroups of $\pi_1(N)$ 
are permuted by automorphisms of $\pi_1(N)$. Details are below.

\begin{proof}[Proof of Theorem~\ref{thm: intro-mostow}]
We first give a proof when $M_1$, $M_2$ are compact,
and later indicate modifications in the finite volume case. 
Let $N_i$ denote the complement in $M_i$ of an 
open $\e$-tubular neighborhood of $S_i$ where $\e$ is smaller than
the normal injectivity radius of $S_i$. 
Contracting along along the geodesic normal to $S_i$
identifies  $M_i\setminus S_i$ with the interior of $N_i$.

Throughout the proof we suppress 
the basepoints while talking about fundamental groups.
All fundamental groups are taken at some basepoints $p_i\in N_i$
with $f(p_1)=p_2$, 
and while talking about the fundamental group of a 
component of $\d N_i$, we implicitly fix an embedded 
path from $p_i$ to each component of $\d N_i$ and 
really are talking of the union of the component 
and the path. The same comment applies to components of $S_i$.

By Section~\ref{sec: rel hyp} the group
$\pi_1(N_i)$ is hyperbolic relative to the fundamental groups
of the components of $\d N_i$. A group is called 
{\it intrinsically elementary} if is not isomorphic to
a non-elementary subgroup of a relatively hyperbolic group.
As noted in~\cite[discussion after Theorem 1.5]{Bel-HP}),
examples of intrinsically elementary groups include
amenable groups and groups with infinite amenable normal subgroups.
Since fundamental groups of the components of $\d N_i$
are intrinsically elementary, and since $N_i$ is aspherical with
incompressible boundary, by a standard argument (spelled out e.g. 
in~\cite[Proof of Theorem 8.1]{Bel-HP}) $f$
induces a homotopy equivalence of pairs 
$(N_1,\d N_1)\to (N_2,\d N_2)$ whose pre/post-composition with
homeomorphisms $M_i\setminus S_i\to\mathrm{Int}(N_i)$ 
is homotopic to $f$. With a slight abuse of notations
we henceforth denote the homotopy equivalence of pairs by $f$. 
We can also assume that in each component of $\d N_1$
the map $f\co (N_1,\d N_1)\to (N_2,\d N_2)$ 
takes homeomorphically one fiber circle 
to a fiber circle in the corresponding components of $\ N_2$
(because fiber circles correspond to maximal 
infinite cyclic normal subgroup which are preserved by any group
isomorphism).

Let $N_i^d$ be the union of $N_i$ and a collection of $2$-discs,
one for each component of $\d N_i$, where each disc is attached along
the fiber circle in the component of $\d N_i$, and where we may assume
that these disks are fibers in disk bundles bounding the chosen circles.
Then $f$ extends to a map $f^d\co N_1^d\to N_2^d$ where 
$f^d$ is a homeomorphism on the corresponding discs. 
The inclusion $N_i\to M_i$ extends to an embedding
$N_i^d\to M_i$ which induces, by Van Kampen theorem, an isomorphism 
of fundamental groups $\pi_1(N_i^d)\cong\pi_1(M_i)$. Then
$f^d$ induces an isomorphism $\psi\co\pi_1(M_1)\to\pi_1(M_2)$, and
a homotopy equivalence $M_1\to M_2$ that extends $f$ up to
homotopy.
Since under the inclusions $N_i^d\to M_i$
the fundamental group of any component
of $\d N_i$ is mapped onto the corresponding component of $S_i$,
we conclude that $\psi$ takes the fundamental group of 
components of $S_1$ to the fundamental groups of corresponding 
components of $S_2$. Fix an arbitrary homotopy equivalence
$\check f\co S_1\to S_2$ induced by $\psi$; like any homotopy
equivalence of closed manifolds $\check f$ is onto.

Let $h\co M_1\to M_2$ be the unique isometry induced by $\psi$ via 
the Mostow rigidity. Suppressing inclusions, 
we think of $h$, $\check f$ as maps from $S_1$ to $M_2$.
These maps are homotopic because they are induced by $\psi$, 
so let $H_t\co S_1\to M_2$ denote the homotopy.
The images $h(S_1)$, $\check f(S_1)=S_2$ 
are totally geodesic embedded submanifolds of $M_2$;
we are to show that $h(S_1)=S_2$. 
Fix an arbitrary component 
$C$ of $S_1$, and lift $H_t\vert_C$ to a homotopy of universal
covers $\tilde H_t\vert_C\co \tilde C\to \tilde M_2$
joining two totally geodesic $(n-2)$-planes in $\tilde M_2=\rhn$.
By compactness of $C$
there is a uniform upper bound on the lengths of 
tracks of $\tilde H_t$, so these totally geodesic
$(n-2)$-planes in $\rhn$ are within
finite Hausdorff distance, hence they coincide, and 
therefore $h(C)=\check f(C)$. 
Applying this argument to each component of $S_1$ gives
$h(S_1)=\check f(S_1)=S_2$.

The uniqueness claim follows from uniqueness in Mostow rigidity.
Indeed, assume there are two isometries 
$h$, $h^\prime$ as above, and let 
$q:=h^\prime\circ h^{-1}$. Then $q$ is 
an isometry of $M_2$ with $q(S_2)=S_2$ such that 
the restriction of $q$ to $M_2\setminus S_2$ is homotopic to identity.
In a smooth manifold any two points can be mapped to each other by 
an isotopic to identity diffeomorphism, and the isotopy can be chosen
equal to identity outside a path joining the points.
Compose $q$ with such self-diffeomorphism $r$ of $M_2$
so that $r\circ q$ is homotopic to identity relative to a basepoint 
$p_2\in M_2\setminus S_2$, and $r\circ q=q$ away from a compact set in 
$M_2\setminus S_2$. 
Now take a loop $\g$ in $M_2$ based at $p_2$, and deform it away
from $S_2$, using that the inclusion $M_2\setminus S_2\to M_2$
is $\pi_1$-surjective. Then $r\circ q$, considered as
a self-map of $M_2\setminus S_2$ takes this loop to a 
homotopic loop based at $p_2$, so $r\circ q$ induces the identity on
of $\pi_1(M_2,p_2)$. Hence $q$ is homotopic to identity, in which
case $q$ equals to the identity by Mostow rigidity. 

In the finite volume case we obtain
$N_i$ by first chopping of the cusps of $M_i$
and then removing an $\e$-neighborhood of $S_i$, and the 
rest of the proof is the same.
\end{proof}

\begin{proof}[Proof of Corollary~\ref{cor: intro-out}] 
immediately follows from 
Theorem~\ref{thm: intro-mostow}, where the last assertion
of Corollary~\ref{cor: intro-out} holds as the isometry group 
of $M$ is finite by Mostow rigidity.
\end{proof}

\begin{rmk}
Finiteness of $\mathrm{Out}(\pi_1(M\setminus S))$ 
can be also deduced from part (3) of 
Theorem~\ref{thm: intro-gr-theoretic-cor}
together with results of Dru{\c{t}}u-Sapir~\cite{DruSap-rips} 
and the author~\cite{Bel-HP} 
(see~\cite[Theorem 1.3]{Bel-HP} for details).
\end{rmk}

\section{Lattices that are relatively hyperbolic}
\label{sec: lat rh}

Most of the ideas needed to prove 
Theorem~\ref{thm: rh lattices} are contained 
in~\cite{ACT, FarWei}. 
The proof uses structure theory of Lie groups, 
various basic facts on relatively
hyperbolic groups, Margulis's Normal Subgroup Theorem, 
Osin's result that any relatively hyperbolic group has 
an infinite normal subgroup of infinite index, 
and Dru{\c{t}}u's result~\cite{Dru-qi} 
that being relatively hyperbolic 
is a quasi-isometry invariant.
Note that Dru{\c{t}}u's result is only used to conclude that a group
that contains a finite index non-elementary relatively hyperbolic 
subgroup is non-elementary relatively hyperbolic. It is applied
in the ``only if'' direction of part (i) after we show that
the finite index subgroup $\L/\L_0$ of $G/G_0$ is relatively hyperbolic,
and in the ``if'' direction of part (ii) after we show that
the finite index subgroup $\L_0$ of $\L$ is relatively hyperbolic.
Thus by making the statement of Theorem~\ref{thm: rh lattices}
slightly less elegant one could avoid referring to~\cite{Dru-qi}.
All the other ingredients of the proof are indispensable.

\begin{rmk}
I expect that Theorem~\ref{thm: rh lattices}
holds for lattices in algebraic groups over local fields,
and encourage an interested reader to work this out. 
As a starting point, in all cases when the Normal Subgroup Theorem 
is known e.g. as in~\cite{Mar-book, BadSha},
the lattices cannot be relatively hyperbolic. 
On the other hand, if $k$ is a non-archimedean local field, then
all lattices in semisimple algebraic $k$-group
of $k$-rank {\it one} are relatively hyperbolic, namely,
by the proof of~\cite[Theorem 7.1]{Lub-rk1}
every lattice is the fundamental group of a finite graph
of group with finite edge groups, and any such group
is hyperbolic relative to the vertex subgroups, as
follows e.g. from~\cite[Definition 2]{Bow-rel}. 
(Nonuniform lattices are infinitely generated, but
most definitions of relative hyperbolicity extend to
infinitely generated case and lead to the same 
classes of groups~\cite{Hru}.)
My interest in describing lattices that are relatively
hyperbolic was motivated by the work
of Behrstock-Dru{\c{t}}u-Mosher,
who by a very different method proved 
in~\cite[Section 13]{BDM} that higher rank 
lattices in semisimple groups over local fields of characteristic zero
are never relatively hyperbolic. 
\end{rmk}

\begin{proof}[Proof of Theorem~\ref{thm: rh lattices}]
It is straightforward to check that
$\L G_0$ is a closed subgroup of $G$, so 
by~\cite[Theorem 1.13]{Rag-book} the group $\L_0:=\L\cap G_0$ is a lattice
in $G_0$. Furthermore, $\L/\L_0$ is a lattice in 
the (discrete) group $G/G_0$
because it leaves invariant
the pushforward of the $G$-invariant measure on $G/\L$
by the projection $G/\L\to (G/G_0)/(\L/\L_0)$; in other words
$\L/\L_0$ is a finite index subgroup of $G/G_0$. 

{\bf Case (i).}
If $G_0$ is compact, then $\L_0$ is finite.
Therefore, 
$\L$ is non-elementary relatively hyperbolic if and only if
$\L/\L_0$ is non-elementary relatively hyperbolic if
and only if $G/G_0$ is non-elementary relatively hyperbolic
(see Appendix~\ref{app: rel hyp}). That relative hyperbolicity of 
$\L/\L_0$ implies relative hyperbolicity of $G/G_0$
depends on delicate work of Dru{\c{t}}u~\cite{Dru-qi}. 

Thus it remains to consider the case when 
$G_0$ is noncompact, in which case we are to prove that
$\Lambda$ is isomorphic to a relatively hyperbolic group
if and only if (ii) holds. First, we treat the ``only if''
part.

{\bf Passing to semisimple quotient.}
Since $G_0$ is noncompact, $\L_0$ is infinite.
Let $R$ be the solvable radical of $G_0$, i.e.
a unique maximal solvable connected normal 
subgroup of $G_0$. Note that $R$ is normal in $G$.
Let $p\co{G}\to {G}/R$ be the projection.
Since $\L$ contains no 
infinite solvable normal subgroup, $R\cap\Lambda$
is finite, and $\L$
is mapped with finite kernel onto a subgroup $p(\L)$ of 
the Lie group ${G}/R$ whose identity component
${G_0}/R$ is semisimple.

If $p(\L)$ is not discrete in ${G}/R$, then by a result of 
Auslander~\cite[Theorem 8.24]{Rag-book} the identity component 
$U$ of the closure of $p(\L)$ is solvable.
Then $p^{-1}(U)$ is a solvable group that has infinite
intersection with $\L$, and it is easy to check that 
the subgroup $\L\cap p^{-1}(U)$ is normal in $\L$,
so this case cannot occur, as $\L$ has no nontrivial solvable 
normal subgroups and $p^{-1}(U)$ is solvable.

Thus $p(\L)$ discrete. Then
it is straightforward to check that
$R\L$ is a closed subgroup of $G$, so 
by~\cite[Theorem 1.13]{Rag-book} the trivial group $R\cap\L$ is a lattice
in $R$, so $R$ is compact.
Recall that a Lie group homomorphism 
with compact kernel takes lattices to lattices,
as can be deduced e.g. from~\cite[Lemma 1.6]{Rag-book},
so $p(\L)$ is a lattice in $G/R$, whose identity component
${G_0}/R$ is semisimple.
 
{\bf Passing to centerless semisimple quotient with no compact factors.}
Let $C$ be a unique maximal connected compact 
normal subgroup of $G_0/R$ so that $H_0:=(G_0/R)/C$ is semisimple with
no compact factors. By uniqueness, $C$ is normal in $G/R$;
we let $H:=(G/R)/C$ and note that $H_0$ is the identity
component of $H$ as $R, C$ are connected.
Since $R$ and $C$ are compact, the kernel of the projection 
$h\co G\to H$ is compact, so the kernel of $\L\to h(\L )$ is finite,
$h(\L)$ is a lattice in $H$, and
$h(\L_0)$ is a lattice in $H_0$. In particular, 
$h(L)$ is still isomorphic to a non-elementary relatively
hyperbolic group.

The center $Z$ of $H_0$ is normal in $H$.
By~\cite[Lemma IX.6.1]{Mar-book}, $Z\cap q(\L_0)$
has finite index in $Z$. On the other hand,
$Z\cap h(\L)$ is an abelian normal subgroup of $h(\L)$,
so $Z\cap h(\L)$ is finite. Thus $Z$ is finite.
We let $Q:=H/Z$ and $Q_0:=H_0/Z$, so that $Q_0$ is
the identity component of $Q$, and $Q_0$ is noncompact
semisimple with trivial center and no compact factors.
Let $K$ be the (compact) kernel of the projection 
$q\co G\to Q$; thus 
the kernel of $\L\to q(\L )$ is finite (and lies in $G_0$ and hence in
$\L_0$),
$\G:=q(\L)$ is a lattice in $Q$, and
$\G_0:=q(\L_0)$ is a lattice in $Q_0$, while 
$\G$ is still isomorphic to a non-elementary relatively
hyperbolic group.
 
{\bf Proving that $G/G_0$ is finite.} 
Since the kernel of $q\co G\to Q$ lies in $G_0$,
it induces an isomorphism $G/G_0\to Q/Q_0$
which takes $\L/\L_0$ onto $\G/\G_0$.
As mentioned above, $\L/\L_0$ is a lattice in  $G/G_0$,
so $\G/\G_0$ is a lattice in $Q/Q_0$, which is discrete,
and hence it suffices to show that $\G/\G_0$ is finite.

Arguing by contradiction suppose that $\G/\G_0$ is infinite.
The center of $\G_0$ is finite, else $\G$ would have an abelian
infinite normal subgroup. 
In fact, it follows from 
Borel Density theorem that any finite normal subgroup
of a lattice in a semisimple Lie group with finite center
and no compact factors lies in the center of the Lie group
(see e.g.~\cite[Corollary 5.42]{Wit-book}). Thus
$\G_0$ is centerless. 
Hence the extension
\[
1\to\G_0\to\G\to \G/\G_0\to 1
\]
is completely determined by the canonical homomorphism
$\G/\G_0\to\mathrm{Out}(\G_0)$, see~\cite[IV, Corollary 6.8]{Bro}.
The homomorphism comes from the $\G$-action on $Q_0$
by conjugation. Since $Q_0$ is semisimple and centerless, 
$\mathrm{Out}(Q_0)$ is finite~\cite[Theorem IX.5.4]{Heg}. 
Thus after passing to a finite index
subgroup, we can assume that $\G$ acts on $Q_0$ by conjugation 
via inner automorphisms of $Q_0$.  
Since the actions preserves $\G_0$, the homomorphism
$\G/\G_0\to\mathrm{Out}(\G_0)$ factors through $N_{Q_0}(\G_0)/\G_0$,
where $N_{Q_0}(\G_0)$ denotes the normalizer of $\G_0$ in $Q_0$.
By the Borel Density Theorem, $N_{Q_0}(\G_0)/\G_0$ is finite
(see e.g.~\cite[Corollary 5.43]{Wit-book}). Hence
after passing to a finite index subgroup 
we can assume that the homomorphism
$\G/\G_0\to\mathrm{Out}(\G_0)$ is trivial, hence
a finite index subgroup of $\G$ is the product of two groups,
commensurable to $\G_0$ and $\G/\G_0$, which are infinite.
This is impossible for a non-elementary relatively
hyperbolic group, which gives a promised contradiction.

{\bf $G_0/K$ has real rank one.}
Note that $\G_0$ is relatively hyperbolic, as a finite index
subgroup of a relatively hyperbolic group $\G$. 
So $\G_0$ is not virtually a product of infinite groups, 
hence $\G_0$ is an irreducible lattice in $Q_0$.
By Proposition~\ref{prop: norm subg rh}, which is
due to Osin, $\G_0$ has an infinite normal subgroup of infinite index,
which is impossible for higher rank lattices by
Margulis's Normal Subgroup Theorem~\cite[Theorem IX.6.14]{Mar-book}.
Thus $Q_0=G_0/K$ has $\mathbb R$-rank one.

{\bf If (ii) holds, then $\L$ is non-elementary relatively hyperbolic.}
The projection $G_0\to G_0/K$ restricted to $\L_0$ has finite kernel
and the image is the lattice $\G_0$ in the isometry group of a hyperbolic space
over reals, quaternions, complex or Cayley numbers.
It is well-known that $\G_0$ is (non-elementary) relatively hyperbolic, 
and hence so is $\L_0$. By~\cite{Dru-qi} $\L$ is (non-elementary)
relatively hyperbolic.
\end{proof}

\appendix

\section{On relatively hyperbolic groups}
\label{app: rel hyp}

Relatively hyperbolic groups were introduced by 
Gromov~\cite{Gro-hgr}, and in this paper we use the 
following version of Gromov's definition developed 
by Bowditch (see~\cite[Definition 1]{Bow-rel}).

Let $G$ be a finitely generated group 
with a (possibly empty) 
family $\mathcal G$ of infinite finitely generated  subgroups.
Suppose that $G$ acts properly discontinuously and isometrically 
on a proper, geodesic, hyperbolic metric space $X$. 
Then $G$ acts on the ideal boundary $\d X$
of $X$ as a convergence 
group~\cite[Proposition 1.12]{Bow-conv}
with limit set $L(G)\subseteq\d X$. 
We refer to~\cite[Section 6]{Bow-rel}, or~\cite{Bow-conv}, 
or~\cite[Section 5]{Yam-rel} or ~\cite{Fre}
for relevant background on 
convergence groups.

\begin{defn}\label{defn: rel hyp}\rm
We say that $G$ is {\it hyperbolic relative to $\mathcal G$} 
if $L(G)=\d X$, each point of $L(G)$ is either conical or bounded
parabolic, and the maximal parabolic subgroups of $G$
are precisely the elements of $\mathcal G$.
\end{defn}
Elements of $\mathcal G$ are also called {\it peripheral subgroups}.
It is known that there only finitely many conjugacy classes of 
peripheral subgroups.

Other definitions of relatively hyperbolic groups 
were developed by Farb~\cite{Far-rel}, 
Bowditch \cite[Definition 2]{Bow-rel}, Yaman~\cite{Yam-rel}, 
Dru{\c{t}}u-Osin-Sapir~\cite{DOS}, Osin \cite{Osi-rel},
Dru{\c{t}}u~\cite{Dru-qi}, and Mineyev-Yaman~\cite{MinYam}. 
It is known that all these definitions are 
equivalent to Definition~\ref{defn: rel hyp},
provided $G$ and all its peripheral subgroups are finitely 
generated and infinite. 
The proofs of various equivalences can be
found in~\cite[Appendix A]{Dah-rel} (cf.~\cite{Bow-rel, Szc, Bum-rel}),
\cite[Theorem 7.10]{Osi-rel}, \cite{Yam-rel}, \cite[Theorem 8.5]{DOS},
\cite[Theorems 4.21,4.34]{Dru-qi}, \cite[Theorem 57]{MinYam}.
With many different approaches to relative hyperbolicity
there are some minor differences in terminology, which
will be pointed out as needed.
 
Bowditch in~\cite{Bow-rel} developed a notion of a relatively
hyperbolic boundary for $G$ that 
is $G$-equivariantly homeomorphic to $\d X$, and such that
any isomorphism of relatively hyperbolic groups that preserves
the collection of peripheral subgroups induces an equivariant 
homeomorphism of the boundaries.

We summarize below some basic properties of 
(discrete)  convergence groups.
A subgroup of a convergence  group is called
{\it elementary} if its limit set 
contains at most two points,
which happens exactly if the subgroup is finite, virtually-$\Z$,
or parabolic. Otherwise, the subgroup is called
{\it non-elementary}. An elementary subgroup containing
a parabolic element cannot contain a hyperbolic element. 
Suppose $H$ is a non-elementary convergence group.
By a ping-pong argument a non-elementary subgroup of $H$ contains 
a nonabelian free group~\cite{Tuk}. 
Any finite normal subgroup of $H$ acts trivially on the limit set,
and the limit set of an infinite normal subgroup of $H$ 
equals to the limit set of $H$ (see~\cite{Fre}). 
In particular, any infinite normal
subgroup of $H$ is non-elementary, and $H$ has no infinite
virtually solvable normal subgroups. 
It follows that $H$
is not virtually a product of infinite groups,
for if $H_1\times H_2$ is a finite index subgroup of $H$
and $H_1, H_2$ are infinite, then 
$H_1\times H_2$, and hence  $H_1, H_2$ are non-elementary,
but the normalizer of any infinite order element of $H_1$ contains $H_2$,
and elementary subgroups have elementary normalizers.

Relative hyperbolicity of $G$ is clearly inherited by extensions 
$\hat G$ of $G$ such that
$\hat G\to G$ has finite kernel,
and subgroups $\check G$ of $G$ of finite index. 
Indeed, one can use the actions of $\hat G$, $\check G$
on the same hyperbolic space induced by the $G$-action, 
and define the peripheral subgroups of $\hat G$, $\check G$
to be stabilizers of the parabolic fixed point of $G$.
Also if $N$ is a finite normal subgroup of $G$, 
then $G/N$ is hyperbolic relative 
to the images of the peripheral subgroups of $G$, because
relative hyperbolicity is encoded in the action on the boundary, 
and the $G$-action on the boundary factors through the projection
$G\to G/N$ (an algebraic proof of this fact
can be found in~\cite[Lemma 4.4]{AMO}).
It follows that the kernel of the action of a relatively hyperbolic group
on its boundary is a unique maximal normal finite subgroup,
which was characterized in purely group theoretic terms 
in~\cite[Lemma 3.3]{AMO}. 
Finally, if $G$ is non-elementary,
then $\hat G$, $\check G$, $G/N$ are non-elementary
because all these groups share the same boundary.

By contrast, is not at all easy
to show that relative hyperbolicity is inherited by
any group $\tilde G$ that contains $G$ as a finite index subgroup.
Dru{\c{t}}u recently proved~\cite{Dru-qi} a much more
general result that the property of being non-elementary relatively 
hyperbolic is invariant under quasi-isometry. 

The following result does not appear in the literature.

\begin{prop}\textup{\bf (Osin)}\label{prop: norm subg rh}
Any non-elementary relatively hyperbolic group has an infinite
normal subgroup of infinite index.
\end{prop}
\begin{proof} The proof below applies to all groups that
are relatively hyperbolic in the sense of Osin, and there are
various differences in terminology, e.g. 
Osin reserves the term ``elementary'' 
for virtually cyclic subgroups, and calls an  element
``hyperbolic'' if it cannot be conjugated into a peripheral subgroup;
an element of infinite order is hyperbolic in Osin's sense
if and only if it fixes exactly two points on the boundary.
In what follows we rephrase Osin's results in the terminology
adopted in this paper.

Take a non-elementary relatively hyperbolic group $G$ with peripheral
subgroups $H_\l$. 
By~\cite[Corollary 4.5]{Osi-elem} $G$ contains an (infinite order)
hyperbolic element. 
By~\cite[Theorem 4.3]{Osi-elem} every (infinite order)
hyperbolic element $g$ is contained
in a unique maximal virtually cyclic subgroup $E(g)$. 
By~\cite[Corollary 1.7]{Osi-elem} $G$ is hyperbolic relative 
to $\{H_\l, E(g)\}$ which is non-elementary,
since the original relatively hyperbolic structure on 
$G$ was non-elementary. Repeating this argument
we find another (infinite order) 
hyperbolic element $h$ such that $G$ is hyperbolic
to  $\{H_\l, E(g), E(h)\}$. By~\cite[Theorem 1.1]{Osi-periph}
for all large $n$ the quotient $\bar G$ of $G$ by the relation $h^n=1$
is hyperbolic relative to the images of $\{H_\l, E(g), E(h)\}$
and the projection map $G\to\bar G$ is one-to-one on each peripheral
subgroups other than $E(h)$. In particular, the group $\bar G$
is infinite, as it contains the infinite virtually cyclic subgroup $E(g)$
(and in fact one can easily show that $\bar G$ is non-elementary,
as at the end of the proof of~\cite[Corollary 1.6]{Osi-periph}).
The kernel is also infinite because it contains the infinite
cyclic subgroup generated by $h^n$, thus the kernel
is an infinite normal subgroup of infinite index.
\end{proof}

\section{Morse theoretic lemma}
\label{app: morse}

Particular cases of the following lemma were used by Toledo~\cite{Tol-nrf}
and also by those working with $3$-hyperbolic manifolds (as explained
to me by Ian Agol).

Suppose that $M$ is a connected complete $n$-manifold of $\sec\le 0$, 
and $S$ is a compact (not necessarily connected) totally geodesic 
submanifold of $M$ of codimension $>1$. 
Denote the universal cover of $M$ by $\tilde M$, and let
$\tilde S$ be the preimage of $S$ under $\tilde M\to M$.
Let $N$ be the complement in $M$ of an 
open tubular neighborhood of $S$.

\begin{lem} \label{lem: morse}
\textup{(i)} $\tilde M\setminus\tilde S$ is diffeomorphic to the manifold
obtained from an open $n$-ball by attaching $k$-handles, one for each
component of $\tilde S$ of codimension $k+1$. \newline
\textup{(ii)} 
$M\setminus S$ is aspherical if and only if each 
component of $S$ has codimension two, in which case
the inclusion $M\setminus S\to M$ induces a surjection
$\pi_1(M\setminus S)\to \pi_1(M)$ whose kernel 
is a free group on the set of components of $\tilde S$.\newline
\textup{(iii)}
no homotopically nontrivial loop in $\d N$ is null-homotopic in $N$.
\end{lem}

\begin{proof}
Since $S$ is compact it has positive injectivity radius $i$.
Let $Z$ be the complement of an open $i$-neighborhood of $\tilde S$
in $\tilde M$. Fix a point $\ast$ in the interior of $Z$. 
The distance to $\ast$ is a smooth Morse function on $Z$.
Indeed, since $\sec(\tilde M)\le 0$, the function has no critical points
on the interior of $Z$. 
By first variation formula the
critical points of the function are the points 
where geodesic rays emanating from $\ast$
intersect $\d Z$ orthogonally. Thus if $C$ is a component 
of $\tilde S$, then the $i$-neighborhood of $C$ contains 
precisely two critical points which lie on a unique ray
from $\ast$ that is orthogonal to $C$.
By Morse theory passing through $C$ results in attaching a handle
that ``goes around'' $C$, thus if $C$ has codimension $k+1$
we attach a $k$-handle, so (i) is proved. 

(ii) Since $S$ has codimension $>1$,
all $k$'s are positive, so that $Z$ is connected and homotopy
equivalent to a wedge of spheres with one $k$-sphere for each 
$k$-handle ``going around'' a component of $\d Z$.
Hence $Z$ is aspherical if and only if $Z$ is homotopy equivalent 
to a wedge of circles, or equivalently if and only if each component
of $S$ has codimension two. In this case
the inclusion $M\setminus S\to M$ clearly induces a surjection
$\pi_1(M\setminus S)\to \pi_1(M)$ whose kernel $\pi_1(Z)$ is 
isomorphic to a free group on the set of components of $\d Z$.

(iii)
If a loop $\alpha\subset\d N$ is null-homotopic in $N$, then $\alpha$ 
lifts to a loop $\tilde\alpha\subset\d Z$ that is null-homotopic in $Z$. 
Every component of $\d Z$ is homotopy equivalent to a circle
that represents a nontrivial free generator in $\pi_1(Z)$.
Thus $\tilde\alpha$ must be null-homotopic in $\d Z$, and
this homotopy can be pushed down to $\d N$.
\end{proof}

\section{Curvature of warped product metrics}
\label{app: components-of-curv-tensor}

In this appendix we review (and also correct!) 
some formulas for the curvature tensor of a  
multiply-warped product metric $dr^2+g_r$ that were 
worked out in~\cite[Section 6]{BW}. 

Suppose $g_r$ is a family of metric on a manifold $F$
where $r$ is on an open interval $I$. 
The computation in~\cite[Section 6]{BW}) works 
provided at each point $w$ of $F$ 
there is a basis of vector fields $\{X_i\}$
on a neighborhood $U_w\subset F$  that is $g_r$-orthogonal
for each $r$. We fix one such a basis for each $w$.
Let $h_i(r)=\sqrt{g_r(X_i,X_i)}$ so that $Y_i=X_i/h_i$
form a $g_r$-orthonormal basis on $U_w$ for any $r>0$.
Since $X_i\neq 0$ and $g_r$ is nondegenerate, $h_i>0$ 

To simplify some of the formulas below 
we denote $g(X,Y)$ by $\langle X , Y\rangle$, denote the vector field
$\frac{\d}{\d r}$ by $\ddr$,  and reserve the notation $\frac{\d}{\d r}T$ 
for the partial derivative of the function $T$ by $r$.

A straightforward tedious computation
(done e.g. in~\cite[Section 6]{BW}) yields the following.
\begin{eqnarray}\label{form: curv of warped prod}
& \langle R_g (Y_i,Y_j) Y_j,Y_i\rangle=
\langle R_{g_r}(Y_i,Y_j) Y_j,Y_i\rangle -
\frac{h_i^\prime h_j^\prime}{h_ih_j},\\
& \langle R_g(Y_i,Y_j) Y_l,Y_m\rangle= 
\langle R_{g_r}(Y_i,Y_j) Y_l,Y_m\rangle
\ \ \ \mathrm{if}\ \{i,j\}\neq \{l,m\},\\
  & \langle R_g(Y_i,\ddr)\ddr ), Y_i \rangle=
-\frac{h_i^{\prime\prime}}{h_i},\ \ \ \ \ \langle
R_g(Y_i,\ddr)\ddr ), Y_j \rangle=0\ \ \ \mathrm{if}\ i\neq j.
\smallskip
\end{eqnarray}

The following mixed term is by far the most 
complicated and is usually the hardest to control.
\begin{lem}\label{lem: correct mixed term}
$2\langle R_g(\ddr ,Y_i) Y_j,Y_k\rangle$
equals to
\begin{eqnarray*}
\langle [Y_i,Y_j],Y_k\rangle
\left(\ln\frac{h_k}{h_j}\right)^\prime
+\langle [Y_k,Y_i],Y_j\rangle 
\left(\ln\frac{h_j}{h_k}\right)^\prime
+\langle [Y_k,Y_j],Y_i\rangle
\left(\ln\frac{h_i^2}{h_jh_k}\right)^\prime .
\end{eqnarray*}
\end{lem}
\begin{proof} 
As noted in~\cite[Section 6]{BW} $\nabla_\ddr Y_i=0$ and 
$[\ddr,Y_i]=-\frac{h_i^\prime}{h_i}Y_i$,
which by the definition of the curvature tensor gives that 
\[
R(\ddr,Y_i)Y_j=\nabla_\ddr\nabla_{Y_i}Y_j + 
\frac{h_i^\prime}{h_i}\nabla_{Y_i}Y_j
\]
Just like in~\cite{BW} Koszul's formula 
can be used to compute that 
\begin{eqnarray}\label{form: nabla_yi yj}
& \nabla_{Y_i} Y_j =-\frac{h_i^\prime}{h_i}\delta_{ij}\ddr 
+\sum_k \frac{Q_{ijk}}{2}Y_k \quad\mathrm{where}\\
& \nonumber Q_{ijk}=
\langle [Y_i,Y_j],Y_k\rangle +
\langle [Y_k,Y_i],Y_j\rangle +
\langle [Y_k,Y_j],Y_i\rangle .
\end{eqnarray}
and then the definition of the curvature tensor
implies that
\begin{eqnarray}\label{form: mixed app}
& 2\langle R_g(\ddr ,Y_i) Y_j,Y_k\rangle=
\frac{h_i^\prime}{h_i} Q_{ijk}+\frac{\d}{\d r} Q_{ijk}.
\end{eqnarray}
Since $\nabla_\ddr Y_k=0$, we know that
$\frac{\d}{\d r} Q_{ijk}$ is the sum of terms of 
the form $\langle\nabla_\ddr [Y_i,Y_j],Y_k\rangle$.
It is noted in~\cite[Section 6]{BW} that 
$[Y_i, Y_j]=\frac{1}{h_ih_j}[X_i, X_j]$
so it is enough to compute $\nabla_\ddr [X_i,X_j]$. 
We can write $[X_i,X_j]=\sum_k c_{ij}^k X_k$, where $c_{ij}^k$
is independent of $r$ (because the identity $[X_i,X_j]=\sum_k c_{ij}^k X_k$
is pullbacked from $TF$ by the projection $F\times I\to F$).
Thus $\frac{\d}{\d r} c_{ij}^k =0$, yet 
one has $c_{ij}^k=\langle [X_i,X_j],X_k\rangle/h_k^2$
because $\langle X_k,X_k\rangle = h_k^2$ .
On the other hand, $\nabla_\ddr Y_k=0$ implies that 
$\nabla_\ddr X_k=\frac{h_k^\prime}{h_k} X_k$,
and in summary $\nabla_\ddr [X_i,X_j]=\sum_k c_{ij}^k\ \nabla_\ddr X_k$, where
\[
c_{ij}^k\ \nabla_\ddr X_k=
c_{ij}^k\ \frac{h_k^\prime}{h_k} X_k =
\langle [X_i,X_j],X_k\rangle\ \frac{h_k^\prime}{h_k^3} X_k=
\langle [Y_i,Y_j],Y_k\rangle\ \frac{h_k^\prime}{h_k} h_i h_j Y_k,
\]
which yields 
\begin{eqnarray}\label{form: app deriv in mixed term}
&\langle\nabla_\ddr  [Y_i,Y_j], Y_k\rangle=
\langle\frac{1}{h_ih_j}\nabla_\ddr [X_i,X_j]+
\left(\frac{1}{h_ih_j}\right)^\prime [X_i, X_j], Y_k\rangle=\\
&\nonumber \left(\frac{h_k^\prime}{h_k} -\frac{h_i^\prime h_j^\prime}{h_ih_j}\right)
\langle [Y_i,Y_j],Y_k\rangle =
\left(\ln\frac{h_k}{h_ih_j}\right)^\prime\langle [Y_i,Y_j],Y_k\rangle, 
\end{eqnarray}
giving the desired formula, after plugging 
(\ref{form: app deriv in mixed term}) into (\ref{form: mixed app})
and collecting terms. 
\end{proof}

\begin{rmk}\label{rmk: bw-bk is ok}
It was stated in~\cite{BK1, BK2} as an immediate implication 
of~\cite[Section 6]{BW} that 
\begin{eqnarray*}
\langle R_g(\ddr ,Y_i) Y_j,Y_k\rangle=
\frac{\left(\ln(h_jh_k)\right)^\prime}{2}
\left(\langle [Y_j,Y_i],Y_k\rangle+\langle
[Y_i,Y_k],Y_j\rangle +\langle [Y_j,Y_k],Y_i\rangle\right),
\end{eqnarray*}
but according to Lemma~\ref{lem: correct mixed term}, 
{\it this formula is incorrect} due to error
in~\cite[Lemma 6.2(1)]{BW}. 
(Note that~\cite[Lemma 6.2(4)]{BW} is also incorrect as a
particular case of~\cite[Lemma 6.2(1)]{BW}).
This erroneous formula was used in~\cite{BW, BK1, BK2} in a crucial way,
yet all the other results in~\cite{BW, BK1, BK2} 
continue to hold without any change with
the correct formula of Lemma~\ref{lem: correct mixed term}, 
because they only depend on the fact that 
$\langle R_g(\ddr ,Y_i) Y_j,Y_k\rangle$ is a 
linear combination of the terms of the form 
$\frac{h_l^\prime}{h_l}\langle [Y_i,Y_j],Y_k\rangle$. 
Specifically, in~\cite{BW} we always
work in the frame $\{Y_k\}$ for which 
$\langle [Y_j,Y_i],Y_k\rangle$ vanishes at the point
where the curvature is computed (see~\cite[Section 7]{BW}), so 
$\langle R_g(\ddr ,Y_i) Y_j,Y_k\rangle$ also vanishes at the point. 
Similarly, in the nearly identical warped product computations 
of~\cite[Section 3]{BK1}, 
and~\cite[Section 10]{BK2} the terms $\frac{h_l^\prime}{h_l}$ 
are uniformly bounded, and
we work with a family of metrics $dr^2+s^2g_r$ 
for which $\langle [Y_j,Y_i],Y_k\rangle$ approaches zero as $s\to\infty$,
so $\langle R_g(\ddr ,Y_i) Y_j,Y_k\rangle$ is 
asymptotically zero, which is what is
needed in~\cite{BK1, BK2}. 
\end{rmk}

\section{Another proof of relative hyperbolicity}
\label{sec: another proof of rh}

This appendix gives another proof of part (1) of
Theorem~\ref{thm: intro-gr-theoretic-cor} via results
of Heintze, Schroeder~\cite{Sch},
Kapovich-Leeb~\cite{KapLee} and Dru{\c{t}}u-Osin-Sapir~\cite{DOS}.

We refer to~\cite{DOS} for basic information on asymptotic cones
and tree-graded spaces.

Suppose that $X$ is a $CAT(0)$ space
with a collection $\{D_i\}$ of open convex subspaces 
of $X$ such that there exists $\e>0$ with the property
that $\e$-neighborhoods of $D_i$'s are disjoint, and
every $2\e$-ball centered outside $\cup_i D_i$ is $CAT(-1)$. 
Fix an arbitrary asymptotic cone $\mathrm{Cone}(X)$ of $X$
(specified by a choice of non-principal ultrafilter, an observation point,
and a sequence of scaling factors), and denote by $\mathcal D$
the collection of all subspaces of $\mathrm{Cone}(X)$ 
that are limits of various subsequences in $\{ D_i\}$.
Suppose that every subspace $D\in\mathcal D$ 
has the following property: given any distinct points
$x,y, z\in\mathcal D$ there is a sequence $z_n\to z$
such that the geodesic triangle with vertices $x, y, z_n$ is open
(where following Kapovich-Leeb we call a geodesic triangle {\it open} 
if any two sides of the triangle intersect only at their common vertex).
Suppose that a group $G$ acts {\it geometrically} (i.e. isometrically,
properly discontinuously, and cocompactly)
on $X$, and that the action permutes $D_i$'s.
With the above assumptions one has
the following.

\begin{thm}\label{thm: kl-ds}
$G$ is hyperbolic relative to the set of stabilizers of 
$D_i$'s.
\end{thm}

\begin{proof}[Proof of Theorem~\ref{thm: kl-ds}]
In~\cite[Section 4.2]{KapLee} it is proved 
that $\mathrm{Cone}(X)$ is tree-graded with respect to $\mathcal D$.
(The result does not need the existence
of a geometric action of $G$ on $X$ that permutes $D_i$'s,
and is stated in~\cite{KapLee}
only when $D_i$'s are uniform neighborhood of flats, 
in which case each $D$ is a flat in $\mathrm{Cone}(X)$, but
the proof in~\cite{KapLee} works without change).

 Since $\e$-neighborhoods of $D_i$'s are disjoint and $G$ acts geometrically,
$\{D_i\}$ falls into finitely many $G$-orbits.
(Else there would exist an infinite sequence of pairwise
$G$-inequivalent $D_k$'s, so take a point in each, pass to quotient, 
choose a converging subsequence, lift it to $X$, and conclude that 
after changing $D_k$'s within their $G$-orbits, 
we can assume that they intersect an arbitrary 
small ball around some point of $X$, contradicting that 
$D_i$'s are $\e$-separated).
Pick a representative in each orbit and denote the representatives
by $D_{i_1},\dots , D_{i_k}$. Denote by $H_j$ the stabilizer
of $D_{i_j}$ in $G$. 

The $H_j$-action on the closure of 
$D_{i_j}$ is geometric. (The action is clearly properly 
discontinuous and isometric, and $D_{i_j}/H_j$ is precompact,
because by construction the inclusion 
$D_{i_j}\to X$ descends to the inclusion $D_{i_j}/H_j\to X/G$,
and $X/G$ is compact).

Fix a base point $\ast$ in $X$, and consider the orbit map
$q_\ast\co G\to X$, which is a quasi-isometry for any word metric on $G$. 
Then there exists $r>0$ such that
$q_\ast(H_{j})$  
is $r$-close to $D_{i_j}$ in the Hausdorff topology for each 
$j=1,\dots , k$.
(For if $\ast_j$ is the nearest point
projection of $\ast$ onto $D_{i_j}$, then the $H_j$-orbit of $\ast_j$ 
is within finite Hausdorff distance from $D_{i_j}$ as well as from
$q_\ast(H_{j})$). 
Therefore, $q_\ast (g H_j)$ and $g(D_{i_j})$ are $r$-close  
in the Hausdorff topology. 

Thus the bi-Lipschitz homeomorphism of asymptotic 
cones $\mathrm{Cone}(G)\to \mathrm{Cone}(X)$ induced by  $q_\ast$ 
maps limits of sequences of left cosets of $H_j$'s homeomorphically onto 
the corresponding elements of $\mathcal D$, which are pieces
of the tree-graded space structure on $X$. For geodesic metric spaces
the property of being tree-graded
is a preserved under or homeomorphism~\cite[Remark 2.18]{DOS},
so $\mathrm{Cone}(G)$ is tree-graded with respect to
limits of sequences of left cosets of $H_j$'s.
By~\cite[Theorem 1.11]{DOS} this
implies that $G$ is hyperbolic relative to the set
of conjugates of $H_1,\dots, H_k$, which equals to the
set of stabilizers of $D_i$'s in $G$.
\end{proof}

Finally, we are ready to give another proof of part (1) of 
Theorem~\ref{thm: intro-gr-theoretic-cor}.
Suppose $M$ is a complete finite volume real hyperbolic $n$-manifold, 
and $S$ is a compact totally geodesic submanifold of codimension two.
As in Remark~\ref{rmk: HSF metrics} 
we give $N$ the Heintze-Schroeder metric with $\sec\le 0$ in which $\d N$
is totally geodesic and  $\sec< 0$ on the interior of $N$.
Let $G:=\pi_1(N)$ and $X$ be the universal cover of $N$,
so that $G$ acts geometrically on the $CAT(0)$ space $X$.
Let $\iota$ be the normal injectivity radius of $\d N$. 
Fix a positive $\e<\frac{\iota}{4}$,  
consider the preimage of the $\e+\frac{\iota}{2}$-neighborhood
of $\d N$ under the covering $X\to N$,
and denote its path-components by $D_i$'s. Since
$\d N$ is locally convex and $X$ is $CAT(0)$, each $D_i$ is convex.
The $\e$-neighborhoods of $D_i$'s are disjoint, and
the sectional curvature is negative on any closed
$2\e$-ball centered outside $\cup_i D_i$, so 
after rescaling we can assume that each
$2\e$-ball centered outside $\cup_i D_i$ is $CAT(-1)$. 
Since $\d N$ is totally geodesic, it is $\pi_1$-incompressible.
Denote the subgroups of $G$ corresponding to the
fundamental groups of 
the path-components of $\d N$ by $H_1,\dots , H_k$;
the stabilizer of each $D_i$ is 
conjugate to some $H_{j(i)}$.
The universal cover of every component of $\d N$
is isometric to $\mathbb R^{n-1}$ or to $\mathbb R\times\rhm$, hence
their limits in $\mathrm{Cone}(X)$ are isometric to an asymptotic
cone of a fixed $D_i$, which is either $\mathbb R^{n-1}$, or 
the product of $\mathbb R$ and a metric tree $T$. 
Applying Lemma~\ref{lem: tree cross R} below,
we see that all assumptions of Theorem~\ref{thm: kl-ds} hold. 

\begin{lem}  \label{lem: tree cross R}
If $T$ is a metric tree, then
for any geodesic triangle in  $\mathbb R\times T$ 
with vertices $x, y, z$ there is a sequence $z_n\to z$
such that the geodesic triangle with vertices $x, y, z_n$ is open.
\end{lem}
\begin{proof}
Denote the projection of $w$ to the $T$-factor by $\bar w$.
If $\bar z$ lies on a line through
$\bar x,\bar y$, then $x,y,z$ lie on a (flat) plane
in $\mathbb R\times T$, so we can just pick $z_n$
outside the line through $x,y$ in the plane. 
Otherwise $\bar z,\bar x,\bar y$ are the endpoints
of a tripod $\Upsilon$; denote by $\bar m$ the
point on $\Upsilon$ of valency $3$. The product
$\mathbb R\times\Upsilon$ is the union of three
flat strips joined at the line $\mathbb R\times\{\bar m\}$;
let $m$ be the point where this line intersects 
the segment $[x,y]$.
It suffices to choose $z_n$ so that the intersection
of $\mathbb R\times\{\bar m\}$ with the segments $[x,y]$, $[x,z_n]$,  $[y,z_n]$
are distinct points, and this can be arranged
by choosing $z_n\to z$ outside the straight lines 
that extend the segments
$[x,m]$, $[y,m]$ into the strip containing $z$.
\end{proof}

{\bf Acknowledgments.}
This paper is dedicated to Thomas Farrell and Lowell Jones
in appreciation of their seminal work on topological rigidity; 
in fact this project was initiated because I got invited to
the 2007 Morelia Conference in honor of Farrell and Jones.
It is a pleasure to thank Ian Agol, Ken Baker, Oleg V. Belegradek,
David Fisher, Vitali Kapovitch, Denis Osin, and the referees 
for discussions and communications
relevant to this work. 

\small
\bibliographystyle{amsalpha}
\bibliography{rh-warp-corrected16july2010}
\end{document}